\documentclass[journal]{IEEEtran}
\IEEEoverridecommandlockouts                           

\usepackage{graphics} 
\usepackage{amsmath} 
\usepackage{amssymb}  
\usepackage{theorem}
\usepackage{color,xspace}
\usepackage{hyperref}
\usepackage{graphicx}
\usepackage{subcaption}
\usepackage{cite}
\usepackage{cleveref}
\usepackage{bm}
\usepackage{bbm}
\usepackage{tikz}
\usepackage{epstopdf} 
\usepackage{url}
\usepackage{multirow}
\usepackage{enumitem}
\usepackage{mathrsfs}  
\usepackage{epstopdf}

\newtheorem{theorem}{Theorem}
\newtheorem{lemma}{Lemma}	

\newtheorem{remark}{Remark}

\newtheorem{proposition}{Proposition}	
	
\newtheorem{definition}{Definition}
\newtheorem{proof}{Proof}    

\def\bbeta{\boldsymbol{\beta}}
\def\balpha{\boldsymbol{\alpha}}

\def\bb{\mathbf{b}}

\def\bx{\mathbf{x}}

\def\bbx{{\ensuremath{\mathbf x}}}
\def\diag{\mathrm{diag}}

\def\SDP{\mathrm{SDP}}
\def\relu{\mathrm{ReLU}}
\def\bE{\mathbf{E}}
\def\bA{\mathbf{A}}
\def\bB{\mathbf{B}}

\def\deepsdp{\mathrm{DeepSDP}}
\def\milp{\mathrm{MILP}}
\def\sdr{\mathrm{SDR}}
\def\lp{\mathrm{LP}}

\renewcommand{\textfloatsep}{5 pt}
\renewcommand{\intextsep}{5 pt}

\newcommand{\mahyar}[1]{{\color{black}{#1}}}

\title{Safety Verification and Robustness Analysis of Neural Networks via Quadratic Constraints and Semidefinite Programming}
\author{Mahyar Fazlyab$^\dagger$, Manfred Morari, George J. Pappas
\thanks{$^\dagger$Corresponding author: mahyarfa@seas.upenn.edu, mahyarfazlyab@jhu.edu. This work was supported by DARPA Assured Autonomy and NSF CPS 1837210. The authors are with the Department of Electrical and Systems Engineering, University of Pennsylvania. Email: \{mahyarfa, morari, pappasg\}@seas.upenn.edu.}}

\begin{document}

\maketitle
\thispagestyle{empty}
\pagestyle{empty}

\begin{abstract} 
	Certifying the safety or robustness of neural networks against input
	uncertainties and adversarial attacks is an emerging challenge in the
	area of safe machine learning and control. To provide such a guarantee,
	one must be able to bound the output of neural networks when their
	input changes within a bounded set. In this paper, we propose a
	semidefinite programming (SDP) framework to address this problem for
	feed-forward neural networks with general activation functions and
	input uncertainty sets. Our main idea is to abstract various properties
	of activation functions (e.g., monotonicity, bounded slope, bounded
	values, and repetition across layers) with the formalism of quadratic
	constraints. We then analyze the safety properties of the abstracted
	network via the S-procedure and semidefinite programming. Our framework
	spans the trade-off between conservatism and computational
	efficiency and applies to problems beyond safety verification. We
	evaluate the performance of our approach via numerical problem instances of various sizes.

\end{abstract}


\section{Introduction} \label{sec:introduction}

Neural networks have become increasingly effective at many difficult machine-learning tasks. However, the nonlinear and large-scale nature of neural networks makes them hard to analyze and, therefore, they are mostly used as black-box models without formal guarantees. In particular, neural networks are highly vulnerable to attacks, or more generally, uncertainty in their input. In the context of image classification, for example, neural networks can be easily deluded into changing their classification labels by slightly perturbing the input image \cite{su2019one}. Indeed, it has been shown that even imperceptible perturbations in the input of the state-of-the-art neural networks cause natural images to be misclassified with high probability \cite{moosavi2017universal}. Input perturbations can be either of an adversarial nature \cite{szegedy2013intriguing}, or they could merely occur due to compression, resizing, and cropping \cite{zheng2016improving}. 
These drawbacks limit the adoption of neural networks in safety-critical applications such as self-driving vehicles \cite{bojarski2016end}, aircraft collision avoidance procedures \cite{julian2016policy}, speech recognition, and recognition of voice commands; see \cite{xiang2018verification} for a survey.

Motivated by the serious consequences of the fragility of neural networks to input uncertainties or adversarial attacks, there has been an increasing effort in developing tools to measure or improve the robustness of neural networks. Many results focus on specific adversarial attacks and attempt to harden the network by, for example, crafting hard-to-classify examples \cite{goodfellow6572explaining, kurakin2016adversarial,papernot2016limitations,moosavi2016deepfool}. Although these methods are scalable and work well in practice, they still suffer from false negatives. Safety-critical applications require provable robustness against any bounded variations in the input data. As a result, many tools have recently  been used, adapted, or developed for this purpose, such as mixed-integer linear programming \cite{bastani2016measuring,dutta2018output,lomuscio2017approach,tjeng2017evaluating}, convex relaxations and duality theory \cite{kolter2017provable,dvijotham2018dual,salman2019convex}, Satisfiability Modulo Theory (SMT) \cite{pulina2012challenging}, dynamical systems \cite{ivanov2018verisig,xiang2018output}, Abstract Interpretation \cite{mirman2018differentiable,gehr2018ai2}, interval-based methods \cite{weng2018towards,zhang2018efficient,hein2017formal,wang2018efficient,wang2018formal}. All these works aim at bounding the worst-case value of a performance measure when their input is perturbed within a specified range.

\medskip

\noindent \textit{Our contribution.} In this paper, we develop a novel framework based on semidefinite programming (SDP) for safety verification and robustness analysis of neural networks against norm-bounded perturbations in their input. Our main idea is to abstract nonlinear activation functions of neural networks by the constraints they impose on the pre- and post- activation values. In particular, we describe various properties of activation functions using Quadratic Constraints (QCs), such as bounded slope, bounded values, monotonicity, and repetition across layers. Using this abstraction, any property (e.g., safety or robustness) that we can guarantee for the abstracted network will automatically be satisfied by the original network as well. The quadratic form of these constraints allows us to formulate the verification problem as an SDP feasibility problem. Our main tool for developing the SDP is the $\mathcal{S}$-procedure from robust control \cite{yakubovich1997s}, which allows us to reason about multiple QCs. Our framework has the following notable features:
\begin{itemize}[leftmargin=*]
	\item We use various forms of QCs to abstract any type of activation function.
	\item Our method can capture \emph{the cross-coupling between neurons across different layers}, thereby reducing conservatism. This feature, which hinges on the assumption that the same activation function is used throughout the entire network (repetition across layers), becomes particularly effective for deep networks. 
	\item We can control the trade-off between computational complexity and conservatism by systematically including or excluding different types of QCs. 
\end{itemize}

In this paper, we focus on the neural network verification problem (formally stated in $\S$\ref{subsec: Problem Statement}) but the proposed framework (input-output characterization of neural networks via quadratic constraints) can be adapted to other problems such as sensitivity analysis of neural networks to input perturbations, output reachable set estimation, probabilistic verification, bounding the Lipschitz constant of neural networks, and closed-loop stability analysis.

%
\subsection{Related Work}
The performance of certification algorithms for neural networks can be measured along three axes. The first axis is the tightness of the certification bounds; the second axis is the computational complexity, and, the third axis is applicability across various models (e.g. different activation functions). These axes conflict. For instance, the conservatism of the verification algorithm is typically at odds with the computational complexity. The relative advantage of any of these algorithms is application-specific. For example, reachability analysis and safety verification applications call for less conservative algorithms, whereas in robust training, computationally fast algorithms are desirable \cite{weng2018towards,kolter2017provable}.

On the one hand, formal verification techniques such as Satisfiability Modulo (SMT) solvers \cite{ehlers2017formal,huang2017safety,katz2017reluplex}, or integer programming approaches \cite{lomuscio2017approach,tjeng2017evaluating} rely on combinatorial optimization to provide tight certification bounds for piece-wise linear networks, whose complexity scales exponentially with the size of the network in the worst-case. A notable work to improve scalability is \cite{tjeng2017evaluating}, where the authors do exact verification of piecewise-linear networks using mixed-integer programming with an order of magnitude reduction in computational cost via tight formulations for non-linearities and careful preprocessing. 

On the other hand, certification algorithms based on continuous optimization are more scalable but less accurate. A notable work in this category is \cite{kolter2017provable}, in which the authors propose a linear-programming (LP) relaxation of piece-wise linear networks and provide upper bounds on the worst-case loss using weak duality. The main advantage of this work is that the proposed algorithm solely relies on forward- and back-propagation operations on a modified network, and thus is easily integrable into existing learning algorithms. In \cite{raghunathan2018certified}, the authors propose an SDP relaxation of one-layer sigmoid-based neural networks based on bounding the worst-case loss with a first-order Taylor expansion. The closest work to the present work is \cite{raghunathan2018semidefinite}, in which the authors propose a semidefinite relaxation (SDR) for certifying robustness of piece-wise linear multi-layer neural networks. This relaxation is based on the so-called ``lifting", where the original problem is embedded in a much larger space. This SDR approach provides tighter bounds than those of \cite{kolter2017provable} but is less scalable. Finally, compared to the SDR method of \cite{raghunathan2018semidefinite}, our SDP framework yield tighter bounds, especially for deep networks, and is not limited to $\relu$ networks. Parts of this work, specialized to probabilistic verification, have appeared in the conference paper \cite{fazlyab2019probabilistic}.


%
%
The rest of the paper is organized as follows. In $\S$\ref{subsec: output reachable set estimation} we formulate the safety verification problem and present the assumptions. In $\S$\ref{sec: Problem Abstraction via Quadratic Constraints}, we abstract the problem with Quadratic Constraints (QCs). In $\S$\ref{sec: SDP for One-layer Neural Networks} we state our main results. In $\S$\ref{sec: Optimization Over the Abstracted Network}, we discuss further utilities of our framework beyond safety verification. In $\S$\ref{sec: Numerical Experiments} we provide numerical experiments to evaluate the performance of our method and compare it with competing approaches. Finally, in $\S$\ref{sec: Conclusions} we draw conclusions.
\subsection{Notation and Preliminaries} \label{subsection: Notation and Preliminaries}
We denote the set of real numbers by $\mathbb{R}$, the set of nonnegative real numbers by $\mathbb{R}_{+}$, the set of real $n$-dimensional vectors by $\mathbb{R}^n$, the set of $m\times n$-dimensional matrices by $\mathbb{R}^{m\times n}$, the $m$-dimensional vector of all ones by $\mathrm{1}_m$, the $m\times n$-dimensional zero matrix by $0_{m\times n}$, and the $n$-dimensional identity matrix by $I_n$. We denote by $\mathbb{S}^{n}$, $\mathbb{S}_{+}^n$, and $\mathbb{S}_{++}^n$ the sets of $n$-by-$n$ symmetric, positive semidefinite, and positive definite matrices, respectively. The $p$-norm ($p \geq 1$) is displayed by $\|\cdot\|_p \colon \mathbb{R}^n \to \mathbb{R}_{+}$. For $A \in \mathbb{R}^{m\times n}$, the inequality $A \geq 0$ is element-wise. For $A \in \mathbb{S}^n$, the inequality $A \succeq 0$ means $A$ is positive semidefinite. For sets $\mathcal{I}$ and $\mathcal{J}$, we denote their Cartesian product by $\mathcal{I} \times \mathcal{J}$. The indicator function of a set $\mathcal{X}$ is defined as $\mathbf{1}_{\mathcal{X}}(x)=1$ if $x \in \mathcal{X}$, and $\mathbf{1}_{\mathcal{X}}(x)=0$ otherwise. For two matrices $A,B$ of the same dimension, we denote their Hadamard product by $A \circ B$. \mahyar{A function $g \colon \mathbb{R}^n \to \mathbb{R}$ is $\alpha$-convex ($0 \leq \alpha < \infty$) if $g-(\alpha/2)\|\cdot\|_2^2$ is convex. Furthermore, $g$ is $\beta$-smooth ($0 \leq \beta<\infty$) if it is differentiable and $(\beta/2)\|\cdot\|_2^2-g$ is convex. Finally, if $g$ is $\alpha$-convex and $\beta$-smooth, then
\begin{gather*} 
\frac{\alpha \beta}{\alpha+\beta}  \|y-x\|_2^2 + \frac{1}{\alpha+\beta}\|\nabla g(y)-\nabla g(x)\|_2^2  \\ \leq  (\nabla g(y)-\nabla g(x))^\top(y-x),
\end{gather*}
for all $x,y \in \mathbb{R}^n$ \cite[Theorem 2.1.12]{nesterov2013introductory}.}
\section{Safety Verification and Robustness Analysis of Neural Networks} \label{subsec: output reachable set estimation}
\subsection{Problem Statement} \label{subsec: Problem Statement}
Consider the nonlinear vector-valued function $f \colon \mathbb{R}^{n_x} \to \mathbb{R}^{n_f}$ described by a multi-layer feed-forward neural network. Given a set $\mathcal{X} \subset \mathbb{R}^{n_x}$ of possible inputs (e.g., adversarial examples), the neural network maps $\mathcal{X}$ to an output set $\mathcal{Y}$ given by
\begin{align}
\mathcal{Y}=f(\mathcal{X}) := \{ y \in \mathbb{R}^{n_f} \mid y=f(x),  \ x \in \mathcal{X}\}.
\end{align}
The desirable properties that we would like to verify can often be represented by a safety specification set $\mathcal{S}_y$ in the output space of the neural network. In this context, the network is safe if the output set lies within the safe region, i.e., if the inclusion $f(\mathcal{X}) \subseteq \mathcal{S}_y$ holds. Alternatively, we can define $\mathcal{S}_x := f^{-1}(\mathcal{S}_y)$ as  the inverse image of $\mathcal{S}_y$ under $f$. Then, safety corresponds to the inclusion $\mathcal{X} \subseteq \mathcal{S}_x$. 


Checking the condition $ \mathcal{Y} \subseteq \mathcal{S}_y$, however, requires an exact computation of the nonconvex set $\mathcal{Y}$, which is very difficult. Instead, our interest is in finding a non-conservative over-approximation ${\mathcal{\tilde Y}}$ of $\mathcal{Y}$ and verifying the safety properties by checking the condition $\tilde{\mathcal{Y}} \subseteq \mathcal{S}_y$. This approach \emph{detects all false negatives} but also produces false positives, whose rate depends on the tightness of the over-approximation--see Figure \ref{fig:safety}. The goal of this paper is to solve this problem for a broad class of input uncertainties and safety specification sets using semidefinite programming.
%
%
\begin{figure}
	\centering
	\includegraphics[width=1.0\linewidth]{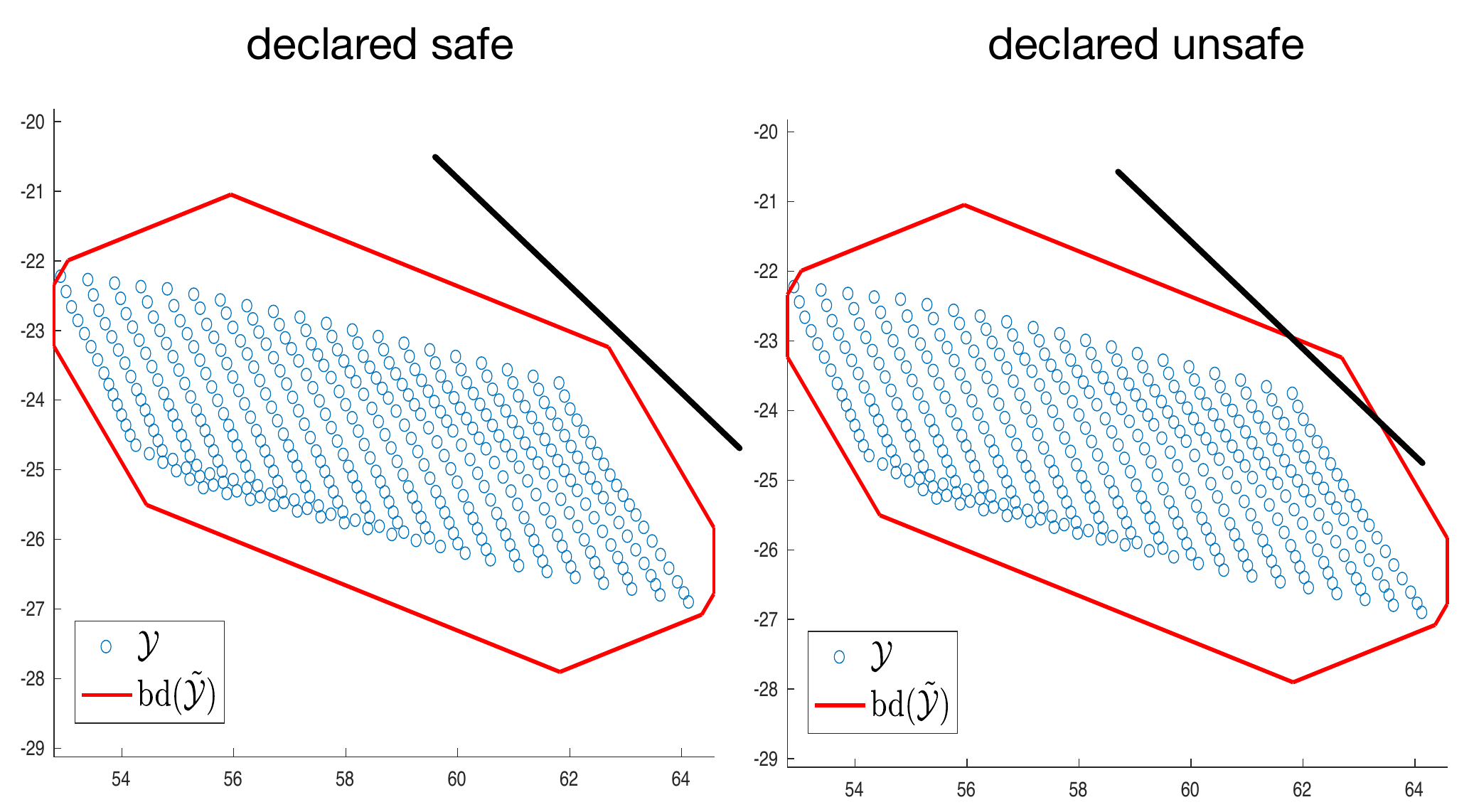}
	\caption{\small The output set (in blue), the boundary of its over-approximation (in red), and the hyperplane characterizing the safe region (in black). Left: The network is deemed safe since $\tilde{\mathcal{Y}} \subseteq \mathcal{S}_y$. Right: The network is deemed unsafe since $\tilde{\mathcal{Y}} \not \subseteq \mathcal{S}_y$.}
	\label{fig:safety}
\end{figure}
\medskip

\subsubsection{Classification Example} Consider a data (e.g., image) classification problem with $n_f$ classes, where a feed-forward neural network $f \colon \mathbb{R}^{n_x} \to \mathbb{R}^{n_f}$ takes as input a data point $x$ and returns an $n_f$-dimensional vector of scores (or logits) -- one for each class. The classification rule is based on assigning $x$ to the class with the highest score.  That is, the class of $x$ is given by $C(x)= \mathrm{argmax}_{1 \leq i \leq n_f} \ f_i(x)$.  To evaluate the local robustness of the neural network around a correctly-classified point $x^\star$, \mahyar{we consider a set $\mathcal{X}$, containing $x^\star$, that represents the set of all possible perturbations of $x^\star$.} 
\mahyar{In image classification, a popular choice are perturbations in the $\ell_{\infty}$ norm, i.e.,  $\mathcal{X} = \{x \colon \|x-x^\star\|_{\infty} \leq \epsilon \}$, where $\epsilon$ is the maximum perturbation applied to each pixel.}
Then the classifier is locally robust at $x^\star$ if it assigns all the perturbed inputs to the same class as $x^\star$, i.e., if $C(x)=C(x^\star)$ for all $x \in \mathcal{X}$. For this problem, the safe set is the polytope
$
\mathcal{S}_y = \{y \in \mathbb{R}^{n_f} \mid y_{i^\star} \geq y_{i} \ \text{for all } i \neq i^\star \},
$
where $i^\star=\mathrm{argmax}_{1 \leq i \leq n_f} \ f_i(x^\star)$ is the class of $x^\star$.

%
%
%

\subsection{Neural Network Model}
For the model of the neural network, we consider an $\ell$-layer feed-forward fully-connected neural network \mahyar{$f \colon \mathbb{R}^{n_x} \to \mathbb{R}^{n_f}$} described by the following recursive equations:
\begin{align} \label{eq: DNN model 0}
x^0 &=x \\ \nonumber 
x^{k+1} &=\phi(W^k x^k + b^k) \quad k=0, \cdots, \ell-1 \\ \nonumber 
f(x) &= W^\ell x^\ell + b^{\ell},
\end{align}
where $x^0 \!=\! x \in \mathbb{R}^{n_0} (n_0\!=\!n_x)$ is the input to the network and $W^{k} \in \mathbb{R}^{n_{k+1} \times n_k}, \ b^k \in \mathbb{R}^{n_{k+1}}$ are the weight matrix and bias vector of the $(k+1)$-th layer. We denote by $n=\sum_{k=1}^{\ell} n_k$ the total number of neurons. The nonlinear activation function $\phi$ ($\relu$\footnote{Rectified Linear Unit.}, sigmoid, tanh,  etc.) is applied coordinate-wise to the pre-activation vectors, i.e., it is of the form
\begin{align} \label{eq: repeated nonlinearity DNN}
\phi(x) := [\varphi(x_1) \ \cdots \ \varphi(x_{n_k})]^\top, \ x \in \mathbb{R}^{n_{k}},
\end{align}
where $\varphi$ is the activation function of each neuron. The output $f(x)$ depends on the specific application we are considering. For example, in image classification with cross-entropy loss, $f(x)$ represents the logit input to the softmax function; or, in feedback control, $x$ is the input to the neural network controller (e.g., tracking error) and $f(x)$ is the control input to the plant. 
\section{Problem Abstraction via Quadratic Constraints} \label{sec: Problem Abstraction via Quadratic Constraints}
In this section, our goal is to provide an abstraction of the verification problem described in $\S$\ref{subsec: Problem Statement} that can be converted into a semidefinite program. Our main tool is Quadratic Constraints (QCs), which were first developed in the context of robust control \cite{megretski1997system} for describing nonlinear, time-varying, or uncertain components of a system. We start with the abstraction of sets using QCs.
%
\subsection{Input Set}
We now provide a particular way of representing the input set $\mathcal{X}$ that will prove useful for developing the SDP.

\begin{definition}\label{def: QC for sets}
	Let $\mathcal{X} \subset \mathbb{R}^{n_x}$ be a nonempty set. Suppose $\mathcal{P}_{\mathcal{X}}$ is the set of all \mahyar{symmetric indefinite} matrices $P$ such that 
	\begin{align} \label{eq: QC for set}
	\begin{bmatrix}
	x \\ 1
	\end{bmatrix}^\top P \begin{bmatrix}
	x \\ 1
	\end{bmatrix} \geq 0  \quad \text{for all } x \in \mathcal{X}.
	\end{align}
	We then say that $\mathcal{X}$ satisfies the QC defined by $\mathcal{P}_{\mathcal{X}}$.
\end{definition} \medskip
Note that by definition, $\mathcal{P}_{\mathcal{X}}$ is a convex cone, i.e., if $P_1,P_2 \in \mathcal{P}_{\mathcal{X}}$ then $\theta_1 P_1 + \theta_2 P_2 \in \mathcal{P}_{\mathcal{X}}$ for all nonnegative scalars $\theta_1,\theta_2$. Furthermore, we can write
\begin{align} \label{eq: set over approximation by qc}
\mathcal{X} \subseteq \bigcap_{P \in \mathcal{P}_{\mathcal{X}}} \left\{x \in \mathbb{R}^{n_x} \colon  \begin{bmatrix}
x \\ 1
\end{bmatrix}^\top P \begin{bmatrix}
x \\ 1
\end{bmatrix} \geq 0 \right\}.
\end{align}
In other words, we can over approximate $\mathcal{X}$ by the intersection of a possibly infinite number of sets defined by quadratic inequalities. 
%
We will see in $\S$\ref{sec: SDP for One-layer Neural Networks} that the matrix $P \in \mathcal{P}_{\mathcal{X}}$ appears as a decision variable in the SDP. In this way, we can optimize the over-approximation of $\mathcal{X}$ to minimize the conservatism of the specific verification problem we want to solve. 

%
\begin{proposition} \label{prop: QC for hyperrect}
	\textbf{(QC for hyper-rectangle)} The hyper-rectangle $\mathcal{X} =  \{x $ $\in \mathbb{R}^{n_x} \mid \underline{x} \leq x \leq \bar{x}\}$ satisfies the QC defined by 
	\begin{gather} \label{eq: QC for hyperrect}
	\mathcal{P}_{\mathcal{X}} \!=\! \left\{P \mid P = \begin{bmatrix}
	-2\Gamma & \Gamma (\underline{x}  + \bar{x}) \\ (\underline{x}+\bar{x})^\top \Gamma   & -2\underline{x}^\top \Gamma \bar{x}
	\end{bmatrix}\right\},
	\end{gather}
	where $\Gamma \in \mathbb{R}^{n_x\times n_x}$ is diagonal and nonnegative. For this set, \eqref{eq: set over approximation by qc} holds with equality. 
\end{proposition}
\begin{proof}
	See Appendix \ref{prop: QC for hyperrect proof}.
\end{proof}

Our particular focus in this paper is on perturbations in the $\ell_{\infty}$ norm, $\mathcal{X} = \{x \mid \|x-x^\star\|_{\infty} \leq \epsilon \}$, which are a particular class of hyper-rectangles with $\underline{x}=x^\star-\epsilon \mathrm{1}$ and $\bar{x}=x^\star+\epsilon \mathrm{1}$. 
We can adapt the result of Proposition \ref{prop: QC for hyperrect} to other sets such as polytopes, zonotopes, and ellipsoids, as outlined below. The derivation of the corresponding QCs can be found in Appendix \ref{QCs for Polytopes, Zonotopes, and Ellipsoids}.

\subsubsection{Polytopes}
Let $\mathcal{X} = \{x \in \mathbb{R}^{n_x} \mid H x \leq h \}$ be a polytope, where $H \in \mathbb{R}^{m \times n_x}, h \in \mathbb{R}^m$. Then $\mathcal{X}$ satisfies the QC defined by
\begin{align}\label{eq: QC for polytope}
\mathcal{P}_{\mathcal{X}} =\left\{ P\mid P \!=\! \begin{bmatrix} H^\top \Gamma H & -H^\top \Gamma h \\ -h^\top \Gamma  H & h^\top \Gamma h\!
\end{bmatrix} \right\},
\end{align}
where $\Gamma \in \mathbb{S}^{m}, \Gamma \geq 0, \Gamma_{ii}=0$. Furthermore, if the set $\{x \in \mathbb{R}^{n_x} \mid Hx \geq h\}$ is empty, then \eqref{eq: set over approximation by qc} holds with equality. 

\medskip

\subsubsection{Zonotopes} A zonotope is an affine transformation of the unit cube, 
%
$\mathcal{X}= \{x \in \mathbb{R}^{n_x} \mid x=x_c + A \lambda, \quad \lambda \in[0,1]^m \}$, 
%
where $A \in \mathbb{R}^{n_x \times m}$ and $x_c \in \mathbb{R}^{n_x}$. Then any $P \in \mathcal{P}_{\mathcal{X}}$ satisfies
\begin{align} \label{eq: QC for zonotopes}
\begin{bmatrix}
A & x_c \\ 0 & 1
\end{bmatrix}^\top 
P
\begin{bmatrix}
A & x_c \\ 0 & 1
\end{bmatrix}+\begin{bmatrix}
2\Gamma & -\Gamma \mathrm{1}_m \\ -\mathrm{1}_{m}^\top \Gamma & 0
\end{bmatrix} \succeq 0,
\end{align}
for some diagonal and nonnegative $\Gamma \in \mathbb{R}^{m\times m}$.
\medskip
\subsubsection{Ellipsoids}
Suppose the input set $\mathcal{X}$ is an ellipsoid defined by
$\mathcal{X} = \{x \in \mathbb{R}^{n_x} \mid \|Ax+b\|_2 \leq 1 \}$, 
where $A \in \mathbb{S}^{n_x}$ and $b \in \mathbb{R}^{n_x}$. Then $\mathcal{X}$ satisfies the QC defined by
\begin{align} \label{eq: Ellipsoid}
\mathcal{P}_{\mathcal{X}} = \left\{P \mid P =  \mu \begin{bmatrix}
-A^\top A & -A^\top b \\ -b^\top A & 1-b^\top b
\end{bmatrix},\ \mu \geq 0 \right\}.
\end{align}
%

%
%
%
%
%
\subsection{Safety Specification Set}
As mentioned in the introduction, the safe set can be characterized either in the output space ($\mathcal{S}_y$) or in the input space ($\mathcal{S}_x$). In this paper, we consider the latter. Specifically, we assume $\mathcal{S}_x$ can be represented (or inner approximated) by the intersection of finitely many quadratic inequalities:
\begin{align} \label{eq: safety set spec}
\mathcal{S}_x = \bigcap_{i=1}^{m} \left\{ x \in \mathbb{R}^{n_x} \mid \begin{bmatrix}
x \\ f(x) \\ 1
\end{bmatrix}^\top S_i \begin{bmatrix}
x \\ f(x) \\ 1
\end{bmatrix} \leq 0 \right\},
\end{align}
where the $S_i \in \mathbb{S}^{n_x+n_f+1}$ are given. In particular, this characterization includes ellipsoids and polytopes in the output space. For instance, for an output safety specification set described by the polytope
%
$\mathcal{S}_y = \cap_{i=1}^{m} \left\{y \in \mathbb{R}^{n_f} \mid c_i^\top y - d_i \leq 0 \right\}$, 
%
the $S_i$'s are given by
\begin{align*}
S_i = \begin{bmatrix}
0& 0 & 0 \\ 0 & 0 & c_i \\ 0 & c_i^\top &  -2d_i
\end{bmatrix} \quad  i=1,\cdots,m.
\end{align*}
%

\subsection{Abstraction of Nonlinearities by Quadratic Constraints}
One of the main difficulties in the analysis of neural networks is the composition of nonlinear activation functions. To simplify the analysis, instead of analyzing the network directly, our main idea is to remove the nonlinear activation functions from the network but retain the constraints they impose on the pre- and post-activation signals. Using this abstraction, any property (e.g., safety or robustness) that we can guarantee for the ``constrained'' network will automatically be satisfied by the original network as well. In the following, we show how we can encode various properties of activation functions (e.g., monotonicity, bounded slope, and bounded values) using QCs. We first provide a formal definition below.
\begin{definition}[QC for functions] \label{eq: QC def 0}
	Let $\phi \colon \mathbb{R}^n \to \mathbb{R}^{n}$ and suppose $\mathcal{Q}_{\phi} \subset \mathbb{S}^{2n+1}$ is the set of all symmetric indefinite matrices $Q$ such that 
	\begin{align} \label{eq: QC def}
	\begin{bmatrix}
	x \\ \phi(x) \\ 1
	\end{bmatrix}^\top Q \begin{bmatrix}
	x \\ \phi(x) \\ 1
	\end{bmatrix} \geq 0 \quad \text{for all } x \in \mathcal{X},
	\end{align}
	where $\mathcal{X} \subseteq \mathbb{R}^{n}$ is a nonempty set. Then we say $\phi$ satisfies the QC defined by $\mathcal{Q}_{\phi}$ on $\mathcal{X}$.
\end{definition}
We remark that our definition of a QC slightly differs from the one used in robust control \cite{megretski1997system}, by including a constant in the vector surrounding the matrix $Q$, which allows us to incorporate affine constraints (e.g., bounded nonlinearities). 
In view of Definition \ref{def: QC for sets}, we can interpret \eqref{eq: QC def} as a QC satisfied by the graph of $\phi$, $\mathcal{G}(\phi):=\{(x,y) \mid y=\phi(x), \ x \in \mathcal{X}\} \subset \mathbb{R}^{2n}$, i.e., $Q_{\phi}=\mathcal{P}_{\mathcal{G}(\phi)}$. Therefore, we can write
\begin{align*}
\mathcal{G}(\phi) \subseteq \bigcap_{Q \in \mathcal{Q}_{\phi}} \left\{(x,y) \in \mathbb{R}^{2n} \colon  \begin{bmatrix}
x \\ y \\1
\end{bmatrix}^\top Q \begin{bmatrix}
x \\ y \\ 1
\end{bmatrix} \geq 0 \right\}.
\end{align*}
In other words, we over-approximate the graph of $\phi$ by a quadratically constrained set. 

The derivation of quadratic constraints is function specific but there are certain rules and heuristics that can be used for all of them which we describe below.

\medskip
\subsubsection{Sector-Bounded Nonlinearities}
Consider the nonlinear function $\varphi \colon \mathbb{R} \to \mathbb{R}$ with $\varphi(0)=0$. We say that $\varphi$  is \emph{sector-bounded} in the sector $[\alpha,\beta]$ ($\alpha \leq \beta<\infty$) if the following condition holds for all $x \in \mathbb{R}$,\footnote{For the case where $\alpha=-\infty$ or $\beta=+\infty$, we define the sector bound inequality as $x (\varphi(x)-\beta x) \leq 0$ and $x(\alpha x -\varphi(x)) \leq 0$, respectively.}
\begin{align} \label{eq: sector bound}
(\varphi(x)-\alpha x)(\varphi(x)-\beta x) \leq 0.
\end{align}
Geometrically, this inequality means that the function $y=\varphi(x)$ lies in the sector formed by the lines $y=\alpha x$ and $y=\beta  x$ (see Figure \ref{fig: sector condition}). As an example, the $\relu$ function belongs to the sector $[0,1]$ and in fact, lies on its boundary.

For the vector case, let $K_1,K_2 \in \mathbb{R}^{n \times n}$ be two matrices such that $K_2-K_1$ is symmetric positive semidefinite. We say that $\phi \colon \mathbb{R}^n \to \mathbb{R}^n$  is sector-bounded in the sector $[K_1,K_2]$ if the following condition holds for all $x \in \mathbb{R}^n$ \cite{khalil2002nonlinear},
\begin{align} \label{eq: sector bound vector}
(\phi(x)-K_1 x)^\top(\phi(x)-K_2 x) \leq 0,
\end{align}
%
%
or, equivalently,
	\begin{align*}
	\begin{bmatrix}
	x \\ \phi(x) \\ 1
	\end{bmatrix}^\top \begin{bmatrix}
	-K_1^\top K_2\!-\!K_2^\top K_1 & K_1^\top\!+\!K_2^\top & 0 \\ K_1\!+\!K_2 & -2I_n & 0 \\ 0 & 0 & 0
	\end{bmatrix}\begin{bmatrix}
	x \\ \phi(x) \\ 1
	\end{bmatrix} \!\geq\! 0.
	\end{align*}
The sector condition does not impose any restriction on the slope of the function. This motivates a more accurate description of nonlinearities with bounded slope \cite{zames1968stability}. 
%


%
\medskip

\subsubsection{Slope-Restricted Nonlinearities}
	A nonlinear function $\phi \colon \mathbb{R}^n \to \mathbb{R}^n$ is slope-restricted in the sector $[\alpha,\beta]$ ($\alpha \leq \beta < \infty$), if for any $x,x^\star \in \mathbb{R}^n$,
	\begin{align} \label{eq: quadratic constraint 0}
	(\phi(x)\!-\!\phi(x^\star)\!-\!\alpha (x\!-\!x^{\star}))^\top(\phi(x)\!-\!\phi(x^{\star})\!-\!\beta (x\!-\!x^{\star})) \leq 0.
	\end{align}
	%
%
For the one-dimensional case $(n = 1)$, \eqref{eq: quadratic constraint 0} states that the chord connecting any two points on the curve of $\phi$ has a slope that is at least $\alpha$ and at most $\beta$: 
\begin{align*} 
\alpha \leq \dfrac{\phi(x)-\phi(x^\star)}{x-x^\star} \leq \beta \quad \forall x,x^\star \in \mathbb{R}.
\end{align*}
%
%
Note that a slope-restricted nonlinearity with $\phi(0)=0$ is also sector bounded.  \mahyar{Furthermore, if $\phi$ is slope-restricted in $[\alpha,\beta]$, then the function $x \mapsto \phi(x+x^\star)-\phi(x^\star)$ belongs to the sector $[\alpha I_n,\beta I_n]$ for any $x^\star$. Finally, the gradient of an $\alpha$-convex and $\beta$-smooth function is slope-restricted in $[\alpha,\beta]$.}
\begin{figure}
	\centering
	\includegraphics[width=0.9\linewidth]{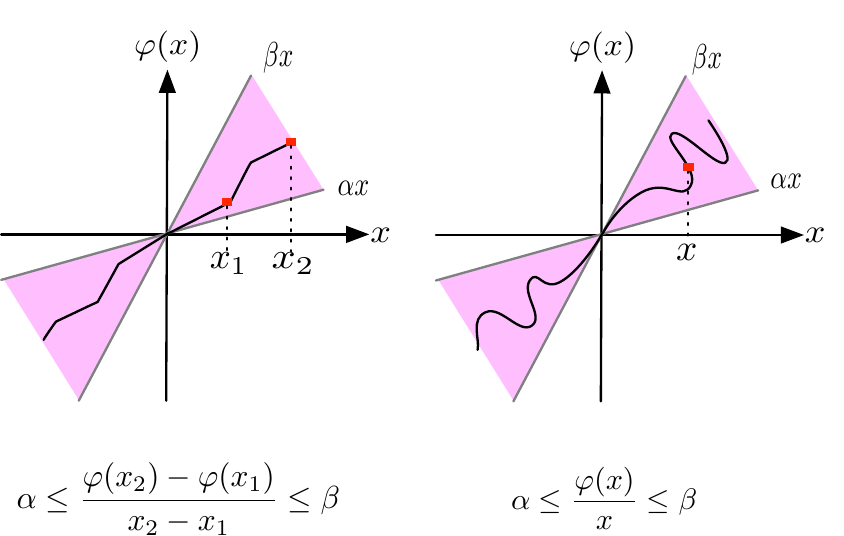}
	\caption{\small A slope-restricted nonlinearity (left) and a sector-bounded nonlinearity (right).}
	\label{fig: sector condition}
\end{figure}

To connect the results of the previous subsection to activation functions in neural networks, we recall the following result {from convex analysis \cite{nesterov2013introductory}.} 
\begin{lemma}[gradient of convex functions] \label{lem: gradient of convex functions}
	Consider a function $g \colon \mathbb{R}^n \to \mathbb{R}$ that is $\alpha$-convex and $\beta$-smooth. Then the gradient function $\nabla g \colon \mathbb{R}^n \to \mathbb{R}^n$ is slope-restricted in the sector $[\alpha,\beta]$.
\end{lemma}
Notably, all \mahyar{commonly-used} activation functions for deep neural networks are gradients of convex functions. Therefore, they belong to the class of slope-restricted nonlinearities, according to Lemma \ref{lem: gradient of convex functions}. We have the following result.
\begin{proposition}\normalfont 
	The following statements hold true.
	\begin{itemize}
		\item[(a)]     The $\relu$ function $\varphi(x)=\max(0,x), \ x \in \mathbb{R}$ is slope-restricted and sector-bounded in $[0,1]$.
		\item[(b)] The sigmoid function, $\varphi(x) = \frac{1}{1+e^{-x}}, \ x \in \mathbb{R}$ is slope-restricted in $[0,1]$.
		\item[(c)] The tanh function, $\varphi(x)=\tanh(x), \ x \in \mathbb{R}$ is slope-restricted and sector-bounded in $[0,1]$.
		\item[(d)] The leaky $\relu$ function, $\varphi(x)=\max(ax,x), \ x \in \mathbb{R}$ with $a>0$ is slope-restricted and sector-bounded in $[\min(a,1),\max(a,1)]$.
		\item[(e)] The exponential linear function (ELU), $\varphi(x) = \max(x,a(e^x-1)),\ x \in \mathbb{R}$ with $a>0$ is slope-restricted and sector-bounded in $[0,1]$.
		\item[(f)] The softmax function, $\phi(x)=  [\frac{e^{x_1}}{\sum_{i=1}^{d} e^{x_i}},$  $\cdots,$  $\frac{e^{x_n}}{\sum_{i=1}^{d} e^{x_i}}]^\top$, $x \in \mathbb{R}^n$ is slope-restricted in $[0,1]$.
	\end{itemize}
\end{proposition}
%
%
%

%
In the context of neural networks, our interest is in \emph{repeated nonlinearities} of the form $\phi(x) = [\varphi(x_1) \ \cdots \ \varphi(x_{n})]^\top.$ Furthermore, the activation values might be bounded from below or above (e.g., the $\relu$ function which outputs a nonnegative value). The sector bound and slope restricted inequalities can become too conservative as they do not capture these properties. In the following, we discuss QCs for these properties.

\subsubsection{Repeated Nonlinearities}
Suppose $\varphi \colon \mathbb{R} \to \mathbb{R}$ is slope-restricted in $[\alpha,\beta]$ ($\alpha \leq \beta$) and let $\phi(x)=[\varphi(x_1)$ $\cdots \varphi(x_n)]^\top$ be a vector-valued function constructed by component-wise repetition of $\varphi$. 
%
%
%
%
%
It is not hard to verify that $\phi$ is also slope-restricted in $[\alpha,\beta]$. Indeed, by summing the slope-restriction conditions
\begin{align*}
(\varphi(x_i)\!-\!\varphi(x_i^\star)-\alpha (x_i\!-\!x_i^{\star}))(\varphi(x_i)\!-\!\varphi(x_i^{\star})\!-\!\beta (x_i\!-\!x_i^{\star})) \!\leq \! 0.
\end{align*}
over $i=1,\cdots,n$, we obtain \eqref{eq: quadratic constraint 0}. However, this representation simply ignores the fact that all the nonlinearities that compose $\phi$ are the same. By taking advantage of this structure, we can refine the quadratic constraint that describes $\phi$. To be specific, for an input-output pair $(x,\phi(x)), \ x \in \mathbb{R}^n$, we can write the inequality
\begin{align} \label{eq: repeated 2}
(\varphi(x_i) \!- \! \varphi(x_j) \!-\!\alpha (x_i\!-\!x_j))(\varphi(x_i) \!- \! \varphi(x_j) \!- \! \beta (x_i \!- \! x_j)) \leq 0,
\end{align}
for all distinct $i,j = 1,\cdots,n,  \ i \neq j$. This particular QC can considerably reduce conservatism, especially for deep networks, as it reasons about \emph{the coupling between the neurons throughout the entire network}. By making an analogy to dynamical systems, we can interpret the neural network as a time-varying discrete-time dynamical system where the same nonlinearity is repeated for all ``time" indexes $k$ (the layer number). Then the QC in \eqref{eq: repeated 2} couples all the possible neurons.
In the following lemma, we characterize QCs for repeated nonlinearities.
\begin{lemma}\emph{\textbf{(QC for repeated nonlinearities)}} \label{lemma: repeated nonlinearities}
	Suppose $\varphi : \mathbb{R} \to \mathbb{R}$ is slope-restricted  in the sector $[\alpha,\beta]$. Then the vector-valued function $\phi(x)=[\varphi(x_1) \ \cdots \varphi(x_n)]^\top$ satisfies the QC
	\begin{align} \label{eq: quadratic constraint}
	\begin{bmatrix}
	x  \\ \phi(x) \\ 1
	\end{bmatrix}^\top  \begin{bmatrix}
	- 2 \alpha \beta T & (\alpha + \beta)T & 0 \\ (\alpha + \beta)T & -2T & 0 \\ 0 & 0 & 0
	\end{bmatrix}  \begin{bmatrix}
	x \\ \phi(x) \\ 1
	\end{bmatrix} \geq 0,
	\end{align}
	for all $x \in \mathbb{R}^n$, where 
	\begin{align} \label{eq: matrix T}
	T= \sum_{1\leq i<j \leq n} \lambda_{ij}(e_i-e_j)(e_i-e_j)^\top, \ \lambda_{ij} \geq 0, 
	\end{align} 
	and $e_i \in \mathbb{R}^n$ is the $i$-th unit vector.
\end{lemma}

\begin{proof}
	By a conic combination of $\binom{n}{2}$ quadratic constraints of the form \eqref{eq: repeated 2}, we obtain \eqref{eq: quadratic constraint}. See Appendix \ref{lemma: repeated nonlinearities proof} for a detailed proof.
\end{proof}
There are several results in the literature about repeated nonlinearities. For instance, in \cite{d2001new,kulkarni2002all}, the authors derive QCs for repeated and odd nonlinearities (e.g. tanh function).
%

\medskip

\subsubsection{Bounded Nonlinearities} Finally, suppose the nonlinear function values are bounded, 
%
%
%
i.e., $\underline \phi \leq \phi(x) \leq \bar \phi$ for all $x \in \mathbb{R}^n$. Using Proposition \ref{prop: QC for hyperrect}, $\phi(x)$ satisfies the quadratic constraint
\begin{align} \label{eq: repeated 4}
\begin{bmatrix}
x \\ \phi(x)  \\ 1
\end{bmatrix}^\top  
\begin{bmatrix} 0 & 0 & 0 \\ 0 & -2D & D(\underline \phi+\bar \phi)\\ 0 & (\underline \phi+\bar \phi)^\top D& -2\underline \phi^\top D\bar \phi \end{bmatrix}  
\begin{bmatrix}
x \\ \phi(x) \\ 1 
\end{bmatrix} \geq 0,
\end{align}
for all $x$, where $D\in \mathbb{R}^{n \times n}$ is diagonal and nonnegative. We can write a similar inequality when the pre- activation values are known to be bounded. More generally, if the graph of $\phi$ is known to satisfy $\mathcal{G}(\phi) \subseteq \mathcal{G}$, then any quadratic constraint for $\mathcal{G}$ is also a valid quadratic constraint for $\phi$.

We observe that the inequalities \eqref{eq: repeated 2}-\eqref{eq: repeated 4} are all quadratic in $(x,\phi(x),1)$, and therefore can be encapsulated into QCs of the form \eqref{eq: QC def}.
%
%
As we show in  $\S$\ref{sec: SDP for One-layer Neural Networks}, the matrix $Q \in \mathcal{Q}_{\phi}$ that abstracts the nonlinearity $\phi$ appears as a decision variable in the SDP.
%

Although the above rules can be used to guide the search for valid QCs for activation functions, a less conservative description of activation functions requires a case-by-case treatment to further exploit the structure of the nonlinearity. 
%
In the next subsection, we elaborate on QCs for $\relu$ activation functions.
\subsection{Quadratic Constraints for $\relu$ Activation Function}
The $\relu$ function precisely lies on the boundary of the sector $[0,1]$. This observation can be used to refine the QC description of $\relu$. Specifically, let $y = \max(\alpha x,\beta x), \ x \in \mathbb{R}^n$ be the concatenation of $n$ $\relu$ activation functions\footnote{For $\relu$, we have $\alpha=0$ and $\beta=1$.}. Then each individual activation function can be described by the following constraints \cite{raghunathan2018semidefinite}:
\begin{align} 
y_i \!=\! \max(\alpha x_i , \beta y_i) \!\iff \! \begin{cases}
(y_i\!-\!\alpha x_i)(y_i\!-\!\beta x_i)\!=\!0 \\  \beta x_i \leq y_i \\ \alpha x_i\leq y_i.
\end{cases}\label{eq: relu qcs}
\end{align}
The first constraint is the boundary of the sector $[\alpha,\beta]$ and the other constraints simply prune these boundaries to recover the $\relu$ function. 
Furthermore, for any two distinct indices $i \neq j$, we can write the constraint \eqref{eq: repeated 2}:
\begin{align} \label{eq: relu qcs 1}
(y_j-y_i-\alpha(x_j-x_i))(y_j-y_i-\beta(x_j-x_i)) \leq 0.
\end{align}
\mahyar{By adding a weighted combination of all these constraints (non-negative weights for inequalities), we find that the function $y = \max(\alpha x,\beta x)$ satisfies
\begin{align} \label{eq: weighted combination}
&\sum_{i=1}^{n}\{\lambda_i (y_i\!-\!\alpha x_i)(y_i-\beta x_i) \!-\! \nu_i (y_i-\beta x_i) - \eta_i (y_i-\alpha x_i)\}+\notag \\
&\sum_{i \neq j} \lambda_{ij} (y_j\!-\!y_i\!-\!\alpha (x_j\!-\!x_i))(y_j\!-\!y_i \!-\! \beta(x_j-x_i))\leq 0,
\end{align}
}%
for all $x \in \mathbb{R}^n$. In the following lemma, we provide a full QC characterization of the $\relu$ function.
\begin{lemma}[Global QC for $\relu$ function]  \label{lem: QC for relu}
	The function $\phi(x) = \max(\alpha x,\beta x)$ satisfies the QC
		\begin{align}  \label{lem: QC for relu 1}
		 \begin{bmatrix}
		 	x \\ \phi(x) \\ 1
		 \end{bmatrix}^\top
		 \begin{bmatrix} Q_{11} & Q_{12} & Q_{13} \\ Q_{12}^\top & Q_{22} & Q_{23} \\ Q_{13}^\top & Q_{23}^\top  & Q_{33}
		\end{bmatrix} 		 \begin{bmatrix}
		x \\ \phi(x) \\ 1
	\end{bmatrix} \geq 0,
	\end{align}
	for all $x \in \mathbb{R}^n$, where
	\begin{align*}
		&Q_{11} = -2 \alpha \beta (\diag( \lambda)+T), \ Q_{12} = (\alpha+\beta)(\diag(\lambda)+T), \\
		&Q_{13} = -\beta \nu- \alpha \eta, \ Q_{22} = - 2(\diag( \lambda)+T), \\ &Q_{23} =  \nu +\eta, \ Q_{33} = 0,
	\end{align*}
	$\nu,\eta \in \mathbb{R}_{+}^n$, and $T$ is given by \eqref{eq: matrix T}.
\end{lemma}
\begin{proof}
	See Appendix \ref{lem: QC for relu proof}.
\end{proof}
\subsubsection{Tightening Relaxations}

The QC of Lemma \ref{lem: QC for relu} holds globally for the whole space $\mathbb{R}^n$. When restricted to a local region $\mathcal{X}$, these QCs can be tightened. Specifically, suppose $y=\max(x,0)$ and define $\mathcal{I}^{+}$, $\mathcal{I}^{-}$, and $\mathcal{I}^{\pm}$ as the set of activations that are known to be always active, always inactive, or unknown for all $x \in \mathcal{X} \subseteq \mathbb{R}^{n}$, i.e.,
\begin{align} \label{eq: neurons indices}
\mathcal{I}^{+} &= \{i  \mid x_i \geq 0 \text{ for all } x \in \mathcal{X}\} \\ \notag
\mathcal{I}^{-} &= \{i  \mid x_i < 0 \text{ for all } x \in \mathcal{X}\} \\ \notag
\mathcal{I}^{\pm}&= \{1,\cdots,n\} \setminus (\mathcal{I}^{+} \cup \mathcal{I}^{-}).
\end{align}
Then the function $y_i=\max(\alpha x_i,\beta x_i)$ belongs to the sector $[\alpha,\alpha],\ [\alpha,\beta]$ and $[\beta,\beta]$ for inactive, unknown, and active neurons, respectively.
%
%
Furthermore, since the constraint $y_i \geq \beta x_i$ holds with equality for active neurons, 
we can write
%
$\nu_i \in \mathbb{R} \text{ if } i \in \mathcal{I}^{+}, \ \nu_i \geq 0 \text{ otherwise}.$
%
Similarly, the constraint $y_i \geq \alpha x_i$ holds with equality for inactive neurons. Therefore, we can write
%
$\eta_i \in \mathbb{R} \text{ if } i \in \mathcal{I}^{-}, \ \eta_i \geq 0 \text{ otherwise}$.
%
%
Finally, the chord connecting the input-output pairs of always-active or always-inactive neurons has slope of $\alpha$ or $\beta$. Equivalently, for any $(i,j) \in (\mathcal{I}^{+} \times \mathcal{I}^{+}) \cup (\mathcal{I}^{-} \times \mathcal{I}^{-})$, we can write
\begin{align*}
(\frac{y_j \!- \!y_i}{x_j \!-\!x_i}-\alpha)(\frac{y_j \!-\!y_i}{x_j \!-\!x_i}\!-\!\beta)=0.
\end{align*}
Therefore, in \eqref{eq: weighted combination}, $\lambda_{ij} \in \mathbb{R}$ for $(i,j) \in (\mathcal{I}^{+} \times \mathcal{I}^{+}) \cup (\mathcal{I}^{-} \times \mathcal{I}^{-})$ and $\lambda_{ij} \geq 0$ otherwise. The above additional degrees of freedom on the multipliers can tighten the relaxation incurred in \eqref{eq: weighted combination}. In the following Lemma, we summarize the above observations.

%
%
%

\begin{lemma}\textbf{(Local QC for $\relu$ function)}  \label{lem: local QC for relu}
	Let $\phi(x) = \max(\alpha x,\beta x),\ x \in \mathcal{X} \subset \mathbb{R}^n$ and define $\mathcal{I}^{+}, \mathcal{I}^{-}$ as in \eqref{eq: neurons indices}. Then $\phi$ satisfies the QC
	\begin{align}  \label{lem: local QC for relu 1}
 \begin{bmatrix}
	x \\ \phi(x) \\ 1
\end{bmatrix}^\top
\begin{bmatrix} Q_{11} & Q_{12} & Q_{13} \\ Q_{12}^\top & Q_{22} & Q_{23} \\ Q_{13}^\top & Q_{23}^\top  & Q_{33}
\end{bmatrix} 		 \begin{bmatrix}
	x \\ \phi(x) \\ 1
\end{bmatrix} \geq 0,
	\end{align}
	for all $x \in \mathcal{X}$, where
	\begin{align*}
	&Q_{11} \!= \!-2 \mathrm{diag}(\balpha \circ \bbeta \circ \lambda)  \!-2\alpha \beta T, \\ &Q_{12} = \mathrm{diag}((\balpha+\bbeta)\circ \lambda)+(\alpha+\beta)T \\
	&Q_{13} = -\bbeta \circ \nu -\balpha \circ \eta,\ 
	Q_{22} = -2T  \\ \notag
	&Q_{23} =  \nu +\eta,\ \notag
	Q_{33} = 0,
	\end{align*}
	with $T=\sum_{1\leq i<j \leq n} \lambda_{ij}(e_i-e_j)(e_i-e_j)^\top$ and 
	\begin{align*}
	\balpha &= [\alpha+(\beta-\alpha)\mathbf{1}_{\mathcal{I}^{+}}(1),\cdots,\alpha+(\beta-\alpha)\mathbf{1}_{\mathcal{I}^{+}}(n)] \\ 
	\bbeta &= [\beta-(\beta-\alpha)\mathbf{1}_{\mathcal{I}^{-}}(1),\cdots,\beta-(\beta-\alpha)\mathbf{1}_{\mathcal{I}^{-}}(n)] \\ 
	\nu_i &\in \mathbb{R}_{+} \text{ for } i \notin \mathcal{I}^{+} \\
	\eta_i &\in \mathbb{R}_{+} \text{ for } i  \notin \mathcal{I}^{-} \\
	\lambda_{ij} &\in \mathbb{R}_{+} \text{ for }  \{i,j\} \notin (\mathcal{I}^{+} \! \times \! \mathcal{I}^{+}) \! \cup \! (\mathcal{I}^{-} \! \times \! \mathcal{I}^{-}).
	\end{align*}
	%
\end{lemma}
\begin{proof}
	See Appendix \ref{lem: QC for relu proof}.
\end{proof}

\mahyar{We do not know \emph{a priori} which neurons are always active or always inactive. However, we can partially find them by computationally cheap presolve steps. Specifically, if $x$ is known to satisfy $\underline{x} \leq x \leq \bar{x}$ (bounds on the pre-activation values), then we have $\mathcal{I}^{+} = \{i \mid \underline{x}_i \geq 0 \}$, $\mathcal{I}^{-} = \{i \mid \bar{x}_i < 0 \}$, and $\mathcal{I}^{\pm} = \{i \mid \bar{x}_i \underline{x}_i \leq 0 \}$. These element-wise bounds can be found by, for example, interval bound propagation \cite{gowal2018effectiveness,cheng2017maximum} or the LP approach of \cite{kolter2017provable}. Indeed, tighter bounds result in a less conservative description of the $\relu$ function outlined in Lemma \ref{lem: local QC for relu}.}

	

\subsection{Other Activation Functions}
Deriving non-conservative QCs for other activation functions (other than $\relu$) is more complicated as they are not on the boundary of any sector. However, by bounding these functions at multiple points by sector bounds of the form \eqref{eq: sector bound}, we can obtain a substantially better over-approximation. In Figure \ref{fig:tanh approx}, we illustrate this idea for the $\tanh$ function.

A secondary approach is to use the element-wise bounds on the inputs to the activation functions to use a tighter sector bound condition in \eqref{eq: sector bound}. For instance, suppose $x \in [\underline{x},\bar{x}] \subseteq \mathbb{R}$. Then the function $\varphi(x)=\tanh(x)$ satisfies the sector condition in \eqref{eq: sector bound}, where $\alpha$ and $\beta$ are given by
\begin{align*}
\alpha &= \begin{cases}
\tanh(\bar{x})/\bar{x}& \text{if } \underline{x} \bar{x} \geq 0 \\
\min(\tanh(\underline{x})/\underline{x},\tanh(\bar{x})/\bar{x}) & \text{otherwise}.
\end{cases} \\ \notag
\beta &= \begin{cases}
\tanh(\underline{x})/\underline{x}& \text{if } \underline{x} \bar{x} \geq 0 \\
1 & \text{otherwise}.
\end{cases}
\end{align*}

More generally, suppose the graph of $\varphi \colon [\underline{x},\bar{x}] \to \mathbb{R}$ is known to satisfy $\mathcal{G}(\varphi) \subseteq \mathcal{G} \subset \mathbb{R}^2$. Then any QC satisfied by $\mathcal{G}$ is also a valid QC for $\varphi$. We can use this property to build local quadratic constraints for general activation functions provided that we can overapproximate their graph locally. This idea is illustrated in Figure \ref{fig:tanh approx} for the case of $\tanh$ function.  


\begin{figure}
	\centering
	\includegraphics[width=\linewidth]{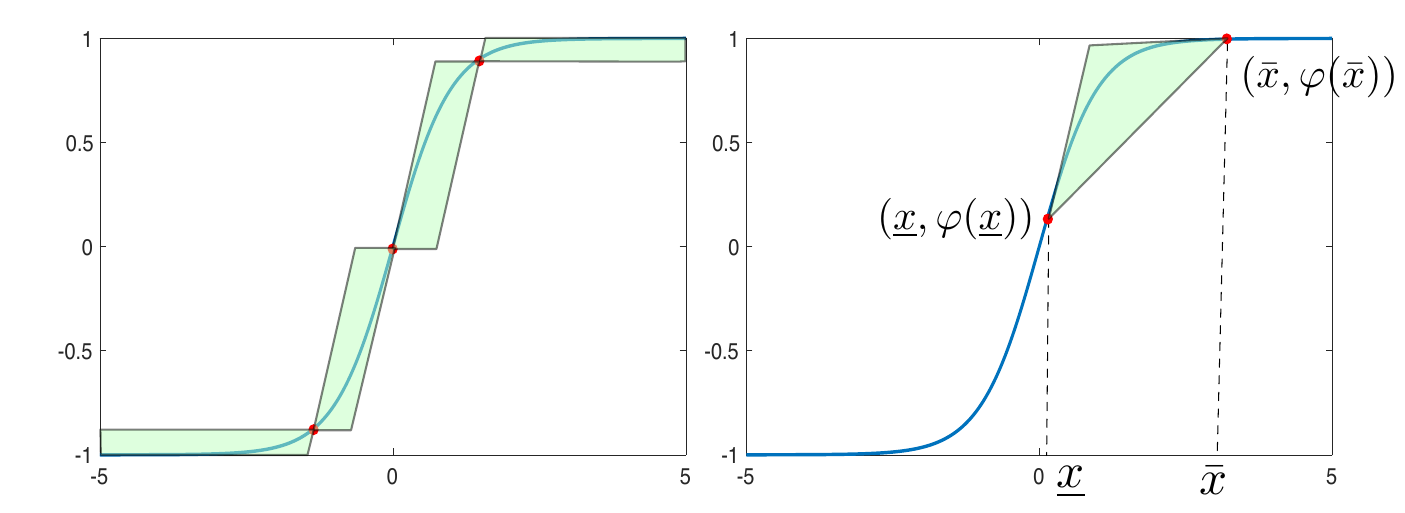}
	\caption{\small (Left) The curve of the tanh function overapproximated on $\mathbb{R}$ by the intersection of three sectors. (Right) The curve of the tanh function overapproximated on $[\underline{x},\bar{x}]$ by a polytope.}
	\label{fig:tanh approx}
\end{figure}

\section{Neural Network Verification Via Semidefinite Programming} \label{sec: SDP for One-layer Neural Networks}
In the previous section, we developed an abstraction of sets and nonlinearities using quadratic constraints. In this section, we use this abstraction to develop an LMI feasibility problem that can assert whether $f(\mathcal{X})\subseteq \mathcal{S}_y$ (or $\mathcal{X}\subseteq \mathcal{S}_x =f^{-1}(\mathcal{S}_y)$). The crux of our idea in the development of the LMI is the $\mathcal{S}$-procedure \cite{yakubovich1997s}, a technique to reason about multiple quadratic constraints, and is frequently used in robust control and optimization \cite{boyd1994linear,ben2009robust}. 
\subsection{Single-layer Neural Networks}
 For the sake of simplicity in the exposition, we start with the analysis of one-layer neural networks and then extend the results to the multi-layer case in $\S$\ref{sec: Multi-layer networks}. We further assume that the safe set $\mathcal{S}_x$ in \eqref{eq: safety set spec} is specified by a single quadratic form, i.e., $m=1$. We state our main result in the following theorem.
\begin{theorem}[SDP for one layer] \label{thm: main result one layer}
	Consider a one-layer neural network $f \colon \mathbb{R}^{n_x} \to \mathbb{R}^{n_f}$ described by the equation
	\begin{align} \label{eq: one layer neural net thm}
	f(x)= W^1 \phi(W^0 x + b^0) + b^1,
	\end{align}
	Suppose $x \in \mathcal{X} \subset \mathbb{R}^{n_x}$, where $\mathcal{X}$ satisfies the QC defined by $\mathcal{P}_{\mathcal{X}}$, i.e., for any $P \in \mathcal{P}_{\mathcal{X}}$, 
	\begin{align} \label{thm: main result one layer 0.5}
	\begin{bmatrix}
	x \\ 1
	\end{bmatrix}^\top P \begin{bmatrix}
	x \\ 1
	\end{bmatrix} \geq 0 \quad \text{for all } x \in \mathcal{X}.
	\end{align}
	Let $\mathcal{Z} = \{z \mid z = W^0 x + b^0, \ x\in \mathcal{X}\}$ and suppose $\phi$ satisfies the QC defined by $\mathcal{Q}_{\phi}$ on $\mathcal{Z}$, i.e., for any $Q \in \mathcal{Q}_{\phi}$,
	\begin{align} \label{thm: main result one layer 1}
		\begin{bmatrix}
			z \\ \phi(z) \\ 1
		\end{bmatrix}^\top Q     \begin{bmatrix}
			z \\ \phi(z) \\ 1
		\end{bmatrix} \geq 0 \quad \text{for all } z \in \mathcal{Z}.
	\end{align}
	%
	%
	Consider the following matrix inequality,
	\begin{align}  \label{thm: main result one layer 2}
	M_{\mathrm{in}}(P) + M_{\mathrm{mid}}(Q) + M_{\mathrm{out}}(S) \preceq 0,
	\end{align}
	where
	\begin{subequations}
		\begin{align}  \label{thm: main result one layer 3}
		M_{\mathrm{in}}(P) &=    \begin{bmatrix}
		I_{n_0} & 0 \\ 0 & 0 \\ 0 & 1
		\end{bmatrix} P \begin{bmatrix}
		I_{n_0} & 0 & 0 \\ 0  & 0 & 1
		\end{bmatrix} \\
		M_{\mathrm{mid}}(Q) &= \begin{bmatrix}
		{W^0}^\top  & 0 & 0 \\ 0 & I_{n_1} & 0 \\ {b^0}^\top & 0 & 1
		\end{bmatrix} Q  \begin{bmatrix}
		{W^0}  & 0 & b^0 \\ 0 & I_{n_1} & 0 \\ 0 & 0 & 1
		\end{bmatrix} \\
		M_{\mathrm{out}}(S) &=\begin{bmatrix}
		I_{n_0} & 0 & 0 \\ 0 & {W^1}^\top & 0 \\ 0 & {b^1}^\top & 1
		\end{bmatrix} S \begin{bmatrix}
		I_{n_0} & 0 & 0\\ 0 & W^1 & b^1 \\ 0 & 0 & 1
		\end{bmatrix},
		\end{align}
	\end{subequations}
	and $S \in \mathbb{S}^{n_x+n_f+1}$ is a given symmetric matrix. If \eqref{thm: main result one layer 2} is feasible for some $P \in \mathcal{P}_{\mathcal{X}}, \ Q \in \mathcal{Q}_{\phi}$, then 
		$$
		 \begin{bmatrix}
		x \\ f(x) \\ 1
		\end{bmatrix}^\top S     \begin{bmatrix}
		x \\ f(x) \\ 1
		\end{bmatrix} \leq 0 \text{ for all } x \in \mathcal{X}.
 		$$
\end{theorem}
\begin{proof} \normalfont
	See Appendix \ref{thm: main result one layer proof}.
\end{proof}

Theorem \ref{thm: main result one layer} states that if the matrix inequality \eqref{thm: main result one layer 2} is feasible for some $(P,Q) \in \mathcal{P}_{\mathcal{X}} \times \mathcal{Q}_{\phi}$, then we can certify that the network $\mathcal{X} \subseteq \mathcal{S}_x$ or $f(\mathcal{X}) \subseteq \mathcal{S}_y$. Since $\mathcal{P}_{\mathcal{X}}$ and $\mathcal{Q}_{\phi}$ are both convex, \eqref{thm: main result one layer 2} is a linear matrix inequality (LMI) feasibility problem and, hence, can be efficiently solved via interior-point method solvers for convex optimization.

\begin{remark}[End-to-end QC for neural network] \normalfont It follows from the proof of Theorem \ref{thm: main result one layer} that, in view of Definition \ref{eq: QC def 0}, the neural network in \eqref{eq: one layer neural net thm} satisfies the QC defined by $(\mathcal{X},\mathcal{Q}_f)$, where 
	\begin{align}
	\mathcal{Q}_f \!=\! \{Q_f \!\mid \! \exists Q \in \mathcal{Q}_{\phi} \text{ s.t. } M_{\mathrm{mid}}(Q) \preceq M_{\mathrm{out}}(Q_f)\}.
	\end{align}
	In other words, for any $Q_f \in \mathcal{Q}_f$ we have
	$$
		\begin{bmatrix}
	x \\ f(x) \\ 1
	\end{bmatrix}^\top Q_f     \begin{bmatrix}
	x \\ f(x) \\ 1
	\end{bmatrix} \geq 0 \quad \text{for all } x \in \mathcal{X}.
	$$
\end{remark}


\subsection{Multi-layer Neural Networks} \label{sec: Multi-layer networks}
We now turn to multi-layer neural networks. Assuming that all the activation functions are the same across the layers (repetition across layers), we can concatenate all the pre- and post-activation signals together and form a more compact representation. To see this, we first introduce $\bbx = [{x^0}^\top  \cdots  {x^{\ell}}^\top]^\top \in \mathbb{R}^{n_0+n}$, where $\ell \geq 1$ is the number of hidden layers. We further define the entry selector matrices $\bE^k \in \mathbb{R}^{n_k \times (n_0+n)}$ such that $x^k = \bE^k \bx$ for $k=0,\cdots,\ell$. Then, we can write \eqref{eq: DNN model 0} compactly as
\begin{subequations} \label{eq: multi layer neural net}
	\begin{align} 
	x \!=\! \mathbf{E}^0 \bbx, \ \bB \bbx  = \phi(\bA \bbx + \bb), \ 
	f(x) \!=\! W^{\ell}\bE^{\ell}\bbx + b^{\ell},
	\end{align}
	where 
	\begin{alignat}{2}
	\bA&= \begin{bmatrix} W^0 & 0 & \cdots & 0 & 0 \\ 0 & W^1 & \cdots & 0 & 0 \\ \vdots & \vdots & \ddots & \vdots & \vdots \\ 0 & 0 & \cdots & W^{\ell-1} & 0 \end{bmatrix}
	\quad \bb &= \begin{bmatrix}
	b^0 \\ b^1 \\ \vdots \\ b^{\ell-1}
	\end{bmatrix} \\
	\bB &= \begin{bmatrix} \  0 & \ \ I_{n_1} & \cdots & 0 & 0 \\ \ \vdots & \ \vdots & \ \ddots & \ \vdots & \vdots \\ \ 0 & \ 0 & \cdots & I_{n_{\ell-1}} & 0 \\ \ 0 & \ 0 & \cdots & 0 & I_{n_{\ell}} \end{bmatrix}. \notag
	\end{alignat}
\end{subequations}
In the following result, we develop the multi-layer counterpart of Theorem \ref{thm: main result one layer} for the multi-layer neural network in \eqref{eq: multi layer neural net}.
\begin{theorem}[SDP for multiple layers] \label{thm: main result multi layer}
	Consider the multi-layer neural network described by \eqref{eq: multi layer neural net}. Suppose $\mathcal{X} \subset \mathbb{R}^{n_x}$ satisfies the QC defined by $\mathcal{P}_{\mathcal{X}}$. Define $\mathcal{Z}= \{\bA\bx +\bb \mid x \in \mathcal{X}\}$ and suppose $\phi$ satisfies the QC defined by $\mathcal{Q}_{\phi}$ on $\mathcal{Z}$.
	Consider the following LMI.
	\begin{align} \label{thm: main result multi layer 1}
	M_{\mathrm{in}}(P) \! + \! M_{\mathrm{mid}}(Q) \!+\! M_{\mathrm{out}}(S) \! \preceq \! 0,
	\end{align}
	where 
	\begin{subequations}
		\begin{align}
		M_{\mathrm{in}}(P) &=\begin{bmatrix}
			\bE^0 & 0 \\  0 & 1
		\end{bmatrix}^\top P \begin{bmatrix}
			\bE^0 & 0 \\  0 & 1
		\end{bmatrix}    \\
		M_{\mathrm{mid}}(Q) &= \begin{bmatrix}
		\bA & \bb \\ \bB & 0 \\ 0 & 1
		\end{bmatrix}^\top Q \begin{bmatrix}
		\bA & \bb \\ \bB & 0 \\ 0 & 1
		\end{bmatrix} \\
		M_{\mathrm{out}}(S) &= \begin{bmatrix}
		\bE^0 & 0 \\ W^{\ell} \bE^{\ell} & b^{\ell} \\ 0 & 1
		\end{bmatrix}^\top S \begin{bmatrix}
		\bE^0 & 0 \\ W^{\ell} \bE^{\ell} & b^{\ell} \\ 0 & 1
		\end{bmatrix},
		\end{align}
	\end{subequations}
	and $S \in \mathbb{S}^{n_x+n_f+1}$ is a given symmetric matrix. If \eqref{thm: main result multi layer 1} is feasible for some $(P,Q) \in \mathcal{P}_{\mathcal{X}} \times \mathcal{Q}_{\phi}$, then
	\begin{align} \label{thm: main result multi layer 2}
		\begin{bmatrix}
	x \\ f(x) \\ 1
	\end{bmatrix}^\top S     \begin{bmatrix}
	x \\ f(x) \\ 1
	\end{bmatrix} \leq 0 \quad \text{for all } x \in \mathcal{X}.
	\end{align}
\end{theorem}
\begin{proof} \normalfont
	See Appendix \ref{thm: main result multi layer proof}.
\end{proof}
\begin{remark} \normalfont  \label{remark: more than one spec}
	For the case that the safe set is characterized by more than one quadratic inequality, i.e., when $m>1$ in \eqref{eq: safety set spec}, then $\mathcal{X} \subseteq \mathcal{S}_x$ if the following LMIs,
	\begin{align}  \label{thm: more than one spec}
	M_{\mathrm{in}}(P_i) + M_{\mathrm{mid}}(Q_i) + M_{\mathrm{out}}(S_i) \preceq 0 \ i=1,\cdots,m,
	\end{align}
	hold for some $P_i \in \mathcal{P}_{\mathcal{X}}$ and $Q_i \in \mathcal{Q}_{\phi}$.
\end{remark}
\section{Optimization Over the Abstracted Network} \label{sec: Optimization Over the Abstracted Network}
In the previous section, we developed an LMI feasibility problem as a sufficient to verify the safety of the neural network. We can incorporate this LMI as a constraint of an optimization problem to solve problems beyond safety verification. Specifically, we can define the following SDP,
\begin{alignat}{2} \label{eq: SDP over abstracted network}
& \mathrm{minimize} && \quad g(P,Q,S) \\
& \mathrm{subject \ to} && \quad M_{\mathrm{in}}(P) + M_{\mathrm{mid}}(Q) + M_{\mathrm{out}}(S) \preceq 0 \notag \\
& && \quad (P,Q,S) \in \mathcal{P}_{\mathcal{X}} \times \mathcal{Q}_{\phi} \times \mathcal{S}, \notag
\end{alignat}
where $g(P,Q,S)$ is a convex function of $P,Q,S$, and $\mathcal{S}$ is a convex subset of $\mathbb{S}^{n_x+n_f+1}$. In the following, we allude to some utilities of the SDP \eqref{eq: SDP over abstracted network}, which we call $\deepsdp$.
\subsection{Reachable Set Estimation} \label{subsec: Certified Upper Bounds}
In Theorem \ref{thm: main result one layer}, we developed a feasibility problem to assert whether $\mathcal{X} \subseteq \mathcal{S}_x$, or equivalently, $f(\mathcal{X}) \subseteq \mathcal{S}_y$. By parameterizing $\mathcal{S}_x$, we can find the best over approximation of $f(\mathcal{X})$ by solving \eqref{eq: SDP over abstracted network}. Suppose $\mathcal{S}_x$ is described by
%
$\mathcal{S}_x = \{x \mid c^\top f(x) - d \leq 0\}$
%
 with a given $c \in \mathbb{R}^{n_f}$ and $d \in \mathbb{R}$. By defining
\begin{align} \label{eq: S for hyperplance}
S= \begin{bmatrix} 0& 0 & 0 \\ 0 & 0 & c \\ 0 &c^\top &  -2d\end{bmatrix},
\end{align}
the feasibility of \eqref{thm: main result multi layer 1} for some $(P,Q) \in \mathcal{P}_{\mathcal{X}} \times \mathcal{Q}_{\phi}$ implies $c^\top f(x) \leq d \quad \text{for all } x \in \mathcal{X}$. In other words, $d$ is a certified upper bound on the optimal value of the optimization problem
%
%
\begin{align} \label{eq: certified upper bound}
\mathrm{maximize} \ c^\top f(x) \quad \mathrm{subject \ to} \ x \in \mathcal{X}.
\end{align}
Now if we treat $d \in \mathbb{R}$ as a decision variable, we can minimize this bound by solving \eqref{eq: SDP over abstracted network} with $g(P,Q,S)=d$. 
This is particularly useful for over approximating the reachable set $f(\mathcal{X})$ by a polyhedron of the form
%
${\mathcal{S}_y} = \cap_{i} \left\{y \in \mathbb{R}^{n_f} \mid c_i^\top y - d_i \leq 0 \right\}$, 
%
where $c_i$ are given and the goal is to find the smallest value of $d_i$, for each $i$, such that $f(\mathcal{X}) \subseteq {\mathcal{S}_y}$.

By reparameterizing $S$ in \eqref{eq: S for hyperplance} we can also compute the best ellipsoidal over-approximation of $f(\mathcal{X})$. Specifically, define
$$
S = \begin{bmatrix}
0 & 0 & 0 \\ 0 & A_y^2 & A_yb_y \\ 0 & b_y^\top A_y & b_y^\top b_y-1
\end{bmatrix}.
$$
Then the inclusion $f(\mathcal{X}) \subseteq \mathcal{S}_y=f(\mathcal{S}_x)$ implies that $f(\mathcal{X})$ is enclosed by the ellipsoid
%
$\mathcal{S}_{y}= \{y \in \mathbb{R}^{n_f} \mid \|A_y y + b_y\|_2 \leq 1\}$.  
%
Therefore, finding the minimum-volume ellipsoid enclosing $f(\mathcal{X})$ amounts to the optimization problem
\begin{alignat}{2} \label{eq: minimum evolume covering ellipsoid}
& \mathrm{minimize} && \quad \log\det(A_y^{-1}) \\
& \mathrm{subject \ to} && \quad M_{\mathrm{in}}(P) + M_{\mathrm{mid}}(Q) + M_{\mathrm{out}}(S(A_y,b_y)) \preceq 0 \notag \\
& && \quad (P,Q,A_y,b_y) \in \mathcal{P}_{\mathcal{X}}\times \mathcal{Q} \times \mathbb{S}^{n_f} \times \mathbb{R}^{n_f}. \notag
\end{alignat}
Note that this problem is not convex in $(A_y,b_y)$ due to the non-affine dependence of $S$ on these variables. However, by using Schur complements, we can formulate an equivalent convex program. We skip the details for the sake of space and refer the reader to \cite{fazlyab2019probabilistic}.
%



\subsection{Closed-Loop Reachability Analysis} \label{subsection:Closed-Loop Reachability Analysis}
By modifying the matrix $S$ in \eqref{eq: S for hyperplance}, we can use a similar approach as presented in $\S \ref{subsec: Certified Upper Bounds}$ to over approximate the reachable sets of closed-loop systems involving neural networks. Specifically, consider a discrete-time Linear Time-Invariant (LTI) system driven by a neural network controller,
\begin{align}
x^{+}  = f_{cl}(x) :=A x + Bf(x), \ x\in \mathcal{X}.
\end{align}
Given a set of current states $\mathcal{X}$, the one-step forward reachable set is
$
\mathcal{X}^{+} = f_{cl}(\mathcal{X}).
$
Suppose $\mathcal{S}_x$ in \eqref{eq: safety set spec} is defined by $\mathcal{S}_x = \{x \in \mathbb{R}^{n_x} \mid c^\top f_{cl}(x) \leq d\}$, where
$$
 S = \begin{bmatrix}
	0 & 0 & A^\top c \\ 0 & 0 & B^\top c\\ c^\top A & c^\top B & -2d
\end{bmatrix}.
$$
According to Theorem \ref{thm: main result multi layer}, the feasibility of the LMI \eqref{thm: main result multi layer 1} for some $(P,Q) \! \in \! \mathcal{P}_{\mathcal{X}} \times \mathcal{Q}_{\phi}$  would allow us to conclude $\mathcal{X} \subseteq \mathcal{S}_x$, or equivalently, $c^\top f_{cl}(x)\leq d$ for all $x \in \mathcal{X}$. By repeating this for different pairs $(c_i,d_i) \in \mathbb{R}^{n_f} \times \mathbb{R},  \ i=1,\cdots,m$, we can overapproximate the one-step reachable set $f_{cl}(\mathcal{X})$ by the polyhedron $\mathcal{P} =\left\{x^{+} \in \mathbb{R}^{n_x} \mid c_i^\top x^{+} - d_i \leq 0 \ i=1,\cdots,m \right\}$. Similarly, we can also overapproximate the closed-loop reachable sets by ellipsoids. In $\S\ref{subsection: Verification of Approximate Model Predictive Control}$ we use this approach to verify a model predictive controller approximated by a neural network.

\begin{figure*}
	\centering
	\includegraphics[width=\linewidth]{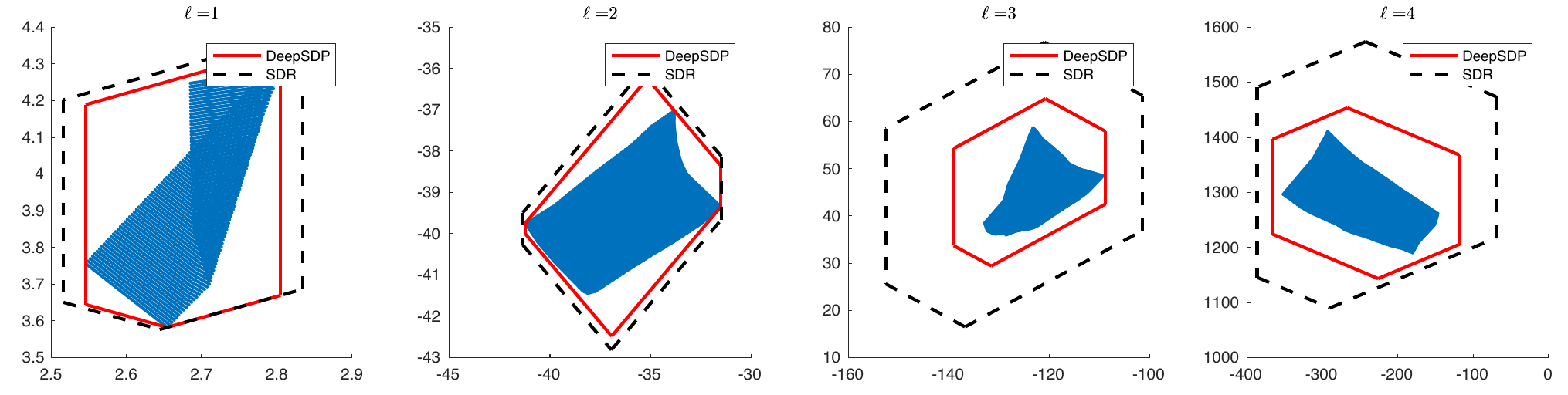}
\caption{\small Illustrations of the output set (blue), the polytope obtained from the results of this paper (red), and the polytope obtained by the semidefinite relaxation of \cite{raghunathan2018semidefinite} (dashed black). The number of neurons per layer is 100, and the input set is the $\ell_{\infty}$ ball with center $x^\star=(1,1)$ and radius $\epsilon=0.1$. The weights of the neural networks are drawn according to the Gaussian distribution $\mathcal{N}(0,1/\sqrt{n_x}).$ From the left to right, the number of hidden layers is $1,2,3,$ and $4$ (the activation function is $\relu$).}
	\label{fig: visualization repeated}
\end{figure*}

\section{Discussion and Numerical Experiments} \label{sec: Numerical Experiments}
In this section, we discuss the numerical aspects of our approach. For solving the SDP, we used MOSEK \cite{mosek} with CVX \cite{cvx} on a 5-core personal computer with 8GB of RAM. For all experiments, we used $\relu$ activation functions and did Interval Bound Propagation as a presolve step to determine the element-wise bounds on the activation functions\footnote{All code, data, and experiments for this paper are available at 
	\url{https://github.com/mahyarfazlyab/DeepSDP}. 
}. We start with the computational complexity of the proposed SDP.

\subsection{Computational Complexity}

\begin{figure}
	\centering
	\includegraphics[width=0.9\linewidth]{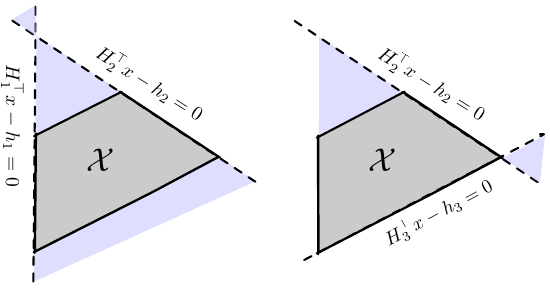}
	\caption{\small A sector bound that is not tight (Left); and a sector bound that is tight (Right).} 
	\label{fig: polytope_qc}
\end{figure}

\subsubsection{Input Set} 
The number of decision variables for the input set depends on the set representation. The quadratically constrained set that over-approximates hyperectangles is indexed by $n_x$ decision variables, where $n_x$ is the input dimension (see Proposition \ref{prop: QC for hyperrect}). Note that for hyper-rectangles, we can include additional quadratic constraints. Indeed, any $x$ satisfying $\underline{x} \leq x \leq \bar{x}$ satisfies $2n_x^2-n_x$ quadratic constraints of the form
$(x_i-\underline{x}_i)(\bar{x}_j-x_i) \geq 0$, $(x_i-\underline{x}_i)(x_j-\underline{x}_j) \geq 0 \ i \neq j$, $(x_i-\bar{x}_i)(x_j-\bar{x}_j) \geq 0 \ i \neq j$.
%
However, one can precisely characterize a hyper-rectangle with only $n_x$ of these quadratic constraints, namely, $(x_i-\underline{x}_i)(\bar{x}_i-x_i)\geq 0 $. Our numerical computations reveal that adding the remaining QCs would not tighten the relaxation.

For polytopes, the maximum number of decision variables is $\binom{m}{2}$, where $m$ is the number of half-spaces defining the polytope. However, we can use some heuristics to remove quadratic constraints that are not ``tight''. For instance, for the polytope $\mathcal{X} = \{x \mid H x \leq h\}$, we can write $\binom{m}{2}$ sector bounds of the form
$
(H_i^\top x -h_i)(H_j^\top x - h_j) \geq 0.
$
Now if the intersection of these hyperplanes belongs to $\mathcal{X}$, then the sector would be tight (see Figure \ref{fig: polytope_qc}). We can verify this by checking the feasibility of
$$
H_i^\top x -h_i = H_j^\top x - h_j = 0, \ H_k^\top x -h_k \leq 0, \ k \neq i,j.
$$
Finally, for the case of ellipsoids, we only have one decision variable, the parameter $\mu$ in \eqref{eq: Ellipsoid}.

\medskip

\subsubsection{Activation Functions} For a network with $n$ hidden neurons, if we use all possible quadratic constraints, the number of decision variables will be $\mathcal{O}(n + n^2)$. If we ignore repeated nonlinearities, we will arrive at $\mathcal{O}(n)$ decision variables. In our numerical experiments, we did not observe any additional conservatism after removing repeated nonlinearities across the neurons of the same layer. However, accounting for repeated nonlinearities was sometimes very effective for the case of multiple layers.
%


\subsubsection{Safety Specification Set} The number of decision variables for the safety specification set depends on how we would like to bound the output set. For instance, for finding a single hyperplane, we have only one decision variable. For the case of ellipsoids, there will be $\mathcal{O}(n_f^2)$ decision variables.

\subsection{Synthetic Examples}
\subsubsection{Number of Hidden Layers} As the first experiment, we consider finding over-approximations of the reachable set of a neural network with a varying number of layers, for a given input set. Specifically, we consider randomly-generated neural networks with $n_x=2$ inputs, $n_f=2$ outputs, and $\ell=\{1,2,3,4\}$ hidden layers, each having $n_k=100$ neurons per layer. For the input set, we consider $\ell_{\infty}$ balls with center $x^\star=(1,1)$ and radius $\epsilon=0.1$. We use $\deepsdp$ to find over-approximations of  $f(\mathcal{X})$ in the form of polytopes (see $\S$\ref{subsec: Certified Upper Bounds}). In Figure \ref{fig: visualization repeated}, we compare the output set $f(\mathcal{X})$ (using exhaustive search over $\mathcal{X}$) with two over-approximations: the red polytope is obtained by solving $\deepsdp$. The dashed black polytope is obtained by the semidefinite relaxation ($\sdr$) approach of \cite{raghunathan2018semidefinite}. We observe that the bounds obtained by $\deepsdp$ are relatively tighter, especially for deeper networks. In Appendix \ref{app: More Visualizations}, we provide more visualizations.
%
%

\begin{figure}
	\centering
	\includegraphics[width=\linewidth]{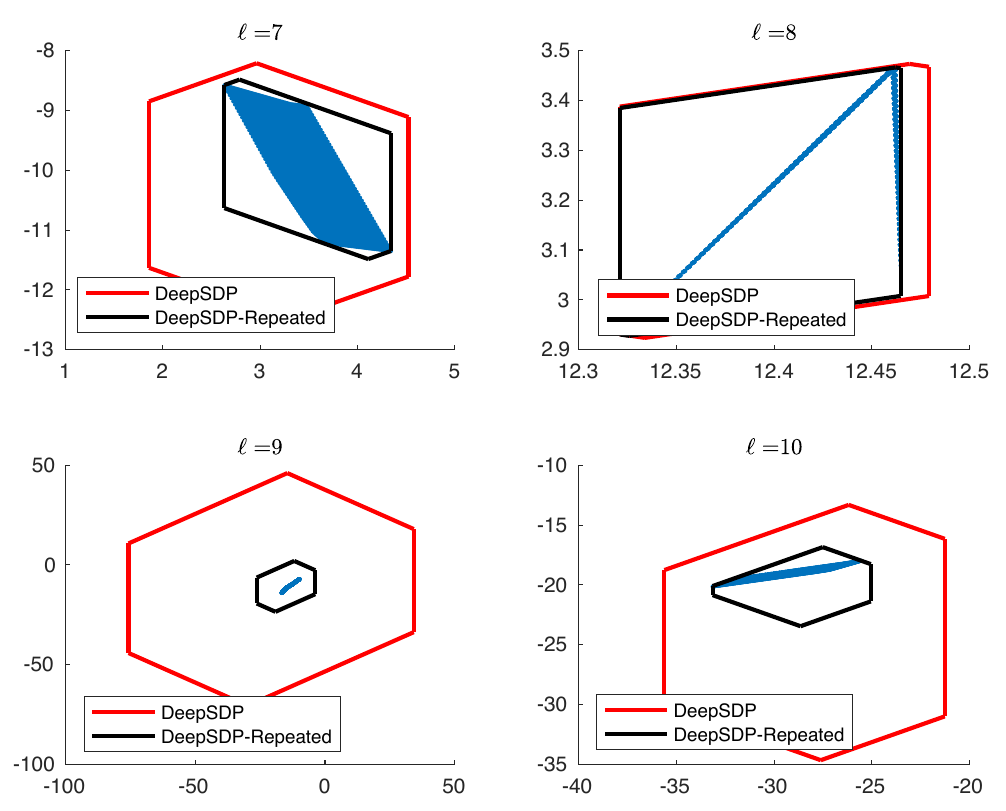}
	\caption{\small Plots of the output set (blue), the polytope obtained from $\deepsdp$ without repeated nonlinearities (red), and the polytope obtained by $\deepsdp$ after including repeated nonlinearities (black).}
	\label{fig: complexity accuracy tradeoff}
\end{figure}

\medskip

\subsubsection{Repeated Nonlinearities} As the second experiment, we study the effect of including repeated nonlinearities on the tightness of the bounds. Specifically, we bound the output of a randomly-generated neural network with $n_x=2$ inputs, $n_f=2$ output, and $n_k=10$ neurons per layer by a polytope with 6 facets. For the input set we consider $\ell_{\infty}$ ball with center $x^\star=(1,1)$ and radius $\epsilon=0.1$. In Figure \ref{fig: complexity accuracy tradeoff}, we plot the output set, and it over-approximation by $\deepsdp$ before and after including repeated nonlinearities. We observe that by including repeated nonlinearities, the bounds become tighter, especially for deep networks.

\begin{figure}
	\centering
	\includegraphics[width=0.9\linewidth]{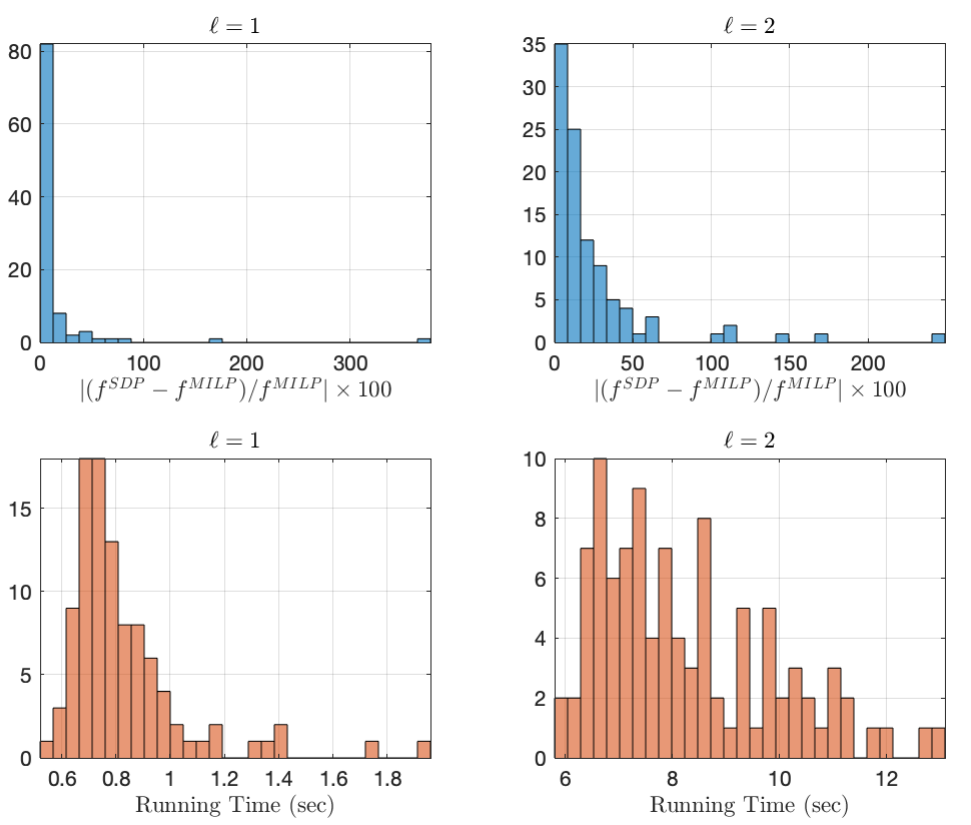}
	\caption{\small (Top) Histograms of the normalized gap between the optimal values and their corresponding bounds obtained by $\deepsdp$. (Bottom) Histograms of solve times in seconds.}
	\label{fig: histograms}
\end{figure}

\medskip

\begin{table*}[]
	\centering
	\begin{tabular}{|l|c|l|c|c|c|c|c|c|c|}
		\hline
		& \multicolumn{5}{c|}{Bounds}                                                                                                                        & \multicolumn{4}{c|}{Running Time (Sec)}                                                                                                   \\ \hline
		$\ell$ & \multicolumn{2}{l|}{$\milp$} & \multicolumn{1}{l|}{$\deepsdp$} & \multicolumn{1}{l|}{$\sdr$} & \multicolumn{1}{l|}{$\lp$} & \multicolumn{1}{l|}{$\milp$} & \multicolumn{1}{l|}{$\deepsdp$} & \multicolumn{1}{l|}{$\sdr$} & \multicolumn{1}{l|}{$\lp$} \\ \hline
		$1$    & \multicolumn{2}{c|}{$1.07$}    & $1.12$                            & $1.13$                        & $1.81$                                & $0.04$                         & $0.82$                            & $0.55$                        &                                \\ \hline
		$2$    & \multicolumn{2}{c|}{$2.04$}    & $2.52$                            & $2.74$                        & $7.62$                                & $25.96$                        & $8.26$                            & $4.71$                        &                                \\ \hline
		$3$    & \multicolumn{2}{c|}{-}           & $11.08$                           & $12.21$                       & $50.60$                               & -                                & $34.18$                           & $31.20$                       &                                \\ \hline
		$4$    & \multicolumn{2}{c|}{-}           & $47.74$                           & $54.15$                       & $368.65$                              & -                                & $78.95$                           & $94.74$                       &                                \\ \hline
		$5$    & \multicolumn{2}{c|}{-}           & $218.8$                             & $266.3$                         & $3004.9$                                & -                                & $164.63$                          & \multicolumn{1}{l|}{$207.77$} &                                \\ \hline
	\end{tabular}
	\caption{Average values (over 100 runs) of different upper bounds for the problem $\sup_{x \in \mathcal{X}} f(x)$ with $\mathcal{X} = \|x-x_{\star}\|_{\infty} \leq \epsilon$, $x_{\star}=\mathrm{1}_{n_x}$ and $\epsilon = 0.2$. The neural network $f$ has $n_x=10$ inputs, $n_f=1$ output and $\ell \in \{1,\cdots,5\}$ hidden layers.}
	\label{tab:average upper bounds}
\end{table*}

\subsubsection{Comparison with Other Methods} As the third experiment, we consider the following optimization problem, 
\begin{align} \label{eq: opt problem experiments}
f^{\star} = \sup_{ \|x-x^\star\|_\infty \leq \epsilon}  c^\top f(x).
\end{align}
To evaluate the tightness of our bounds, we compare $\deepsdp$ with the $\milp$ formulation of \cite{dutta2018output}, the semidefinite relaxation ($\sdr$) of \cite{raghunathan2018semidefinite}, and the LP relaxation of \cite{kolter2017provable}.
For the problem data, we generated random instances of neural networks with $n_x=10$ inputs, $n_f=1$ output and $\ell \in \{1,\cdots,5\}$ hidden layers; for each layer size, we generated $100$ random neural networks with their weights and biases chosen independently from the normal distribution $\mathcal{N}(0,1/\sqrt{n_x})$. For the input set, we consider $x^\star=\mathrm{1}_{n_x}$ and $\epsilon=0.2$. In Table \ref{tab:average upper bounds}, we report the comparisons of bounds and running times. The MILP formulation finds the global solution but the running time grows quickly as the number of neurons increases. Compared to $\sdr$, the bounds of $\deepsdp$ are relatively tighter, especially for deeper networks. Finally, the LP relaxation bounds are considerably looser but the running time is negligible. In Figure \ref{fig: histograms}, we plot the histograms of the normalized gap between the optimal value $f^\star$ (obtained by MILP) and the upper bound $f^{\SDP}$ for layer sizes $\ell=1,2$. 
\subsection{Verification of Approximate Model Predictive Control} \label{subsection: Verification of Approximate Model Predictive Control}
\begin{figure*}
	\centering
	\includegraphics[width=1.0\linewidth]{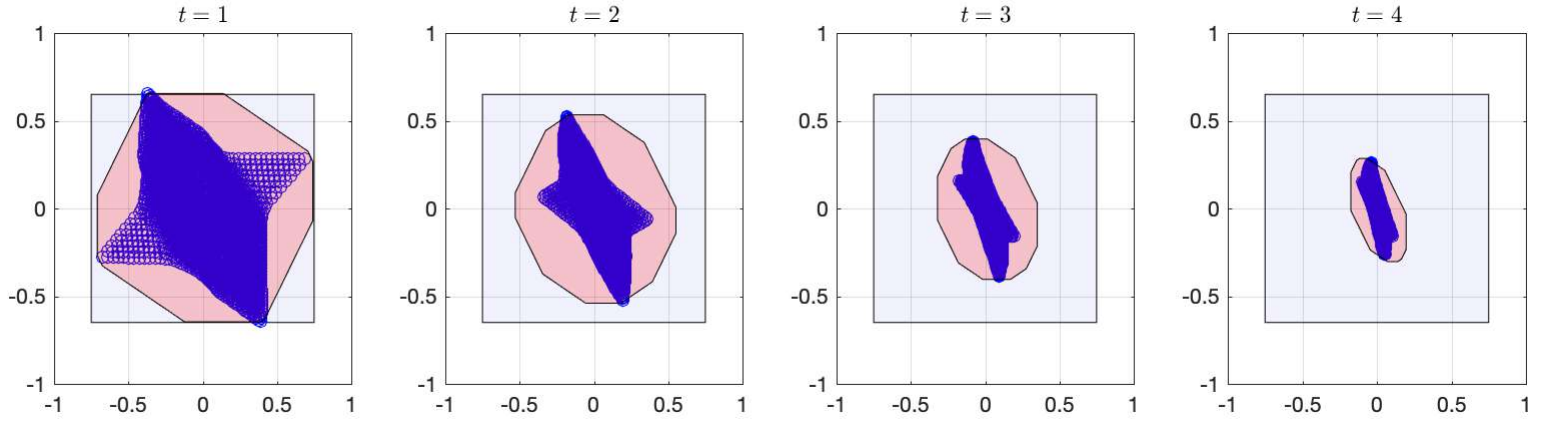}
	\caption{\small Illustration of the invariant set $\mathcal{E}$ (light blue), the output reachable sets (dark blue) and their over-approximations (light red) for the system described in $\S$ \ref{subsection: Verification of Approximate Model Predictive Control}. To over approximate the reachable set at each time step $t$, we use the over-approximation of the reachable set computed by $\deepsdp$ at $t-1$ as the initial set.}
	\label{fig: nn_mpc_1}
\end{figure*}

\begin{figure}
	\centering
	\includegraphics[width=1\linewidth]{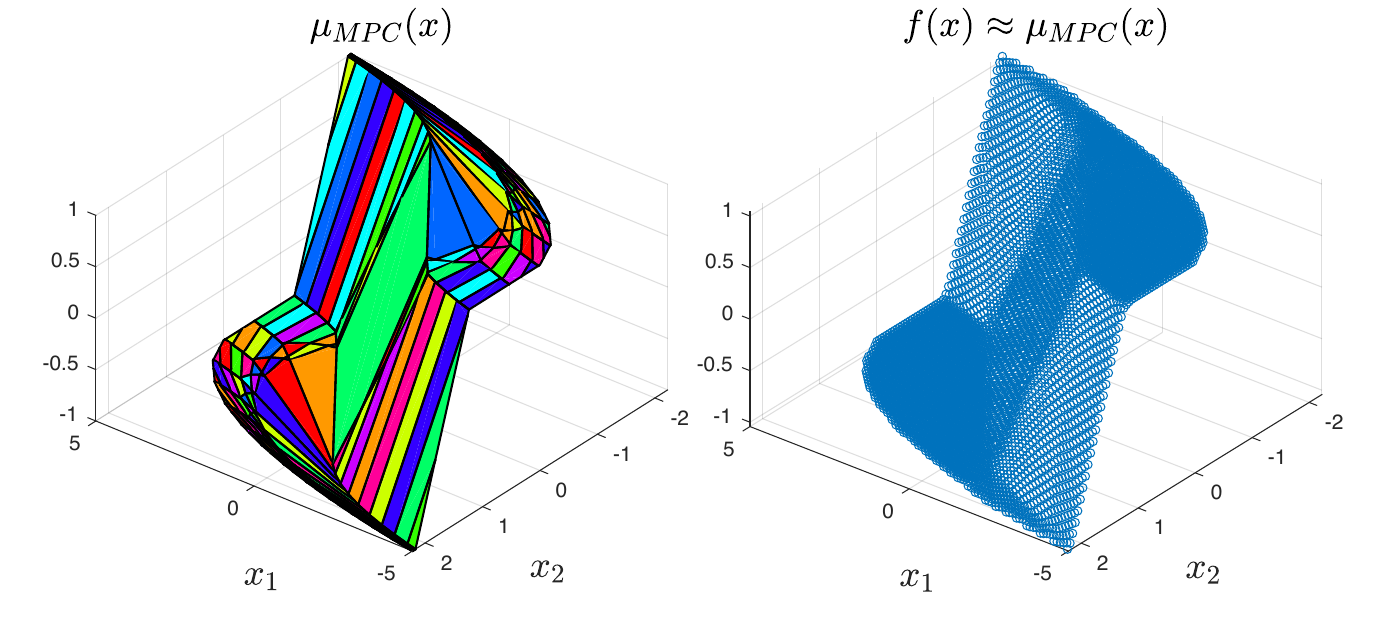}
	\caption{\small The explicit MPC control law for the system described in $\S$\ref{subsection: Verification of Approximate Model Predictive Control} (left), and its approximation by a neural network (right).}
	\label{fig: mpc and nn policy}
\end{figure}

Consider an LTI system
\begin{align} \label{eq: LTI+NN 0}
x_{k+1} &= A x_k + B u_k,  \ x_k \in \mathcal{X}, \ u_k \in \mathcal{U},
\end{align} 
where $x_k \in \mathbb{R}^{n_x}$ is the state at time $k$, $u_k \in \mathbb{R}^{n_u}$ is the control input, and $A,B$ are matrices of appropriate size. The state and control input are subject to the box constraints $\mathcal{X} = \{x \mid \underline{x} \leq x \leq \bar{x}\}$ and $\mathcal{U} = \{u \mid \underline{u} \leq u \leq \bar{u}\}$. 

Suppose the control policy is parameterized by a multi-layer fully-connected feed-forward network $f$ that is trained off-line to approximate a model predictive control (MPC) law $\mu^\star(x)$. The motivation is to reduce the computational burden of solving an optimization problem online to determine the MPC control action. The trained neural network, however, does not necessarily satisfy the specifications of the MPC control law such as state and control constraint satisfaction. To ensure input constraint satisfaction, we project the neural network output onto $\mathcal{U}$, resulting in the closed-loop system,
\begin{align} \label{eq: LTI+NN}
x_{k+1} &=f_{cl}(x_k):= A x_k + B \mathrm{Proj}_{\mathcal{U}}(f(x_k)).
\end{align} 
Note that for input box constraints, $\mathcal{U}=\{u \mid \underline{u} \leq u \leq \bar{u}\}$, we can embed the projection operator as two additional layers with a specific choice of weights and biases. Indeed, for an $\ell$-layer $f$, we can describe $f_p(x)=\mathrm{Proj}_{\mathcal{U}}(f(x_k))$ via the $(\ell+2)$-layer $\relu$ network,
\begin{align} \label{eq: DNN model with proj}
x^0 &=x \\ \nonumber 
x^{k+1} &=\max(W^k x^k + b^k,0) \quad k=0, \cdots, \ell-1 \\ \nonumber 
x^{\ell+1} &= \max(W^{\ell}x^{\ell}+b^{\ell}-\underline{u},0) \\ \nonumber
x^{\ell+2} &= \max(-x^{\ell+1}+\bar{u}-\underline{u},0)\\ \nonumber
f_p(x) &= -x^{\ell+2}+\bar{u}.
\end{align}
To validate state constraint satisfaction, we must ensure that there is a set of initial states $\mathcal{E}\subseteq \mathcal{X}$ whose trajectories would always satisfy the state constraints. One such set is a positive invariant set. By definition, a set $\mathcal{E}$ is positively invariant with respect to $f_{cl}$, if and only if $x_0 \in \mathcal{E}$ implies $x_k \in \mathcal{E}$ for all $k \geq 1$. Equivalently, $\mathcal{E}$ is positively invariant if $f_{cl}(\mathcal{E}) \subseteq \mathcal{E}$. We now show that how we can compute a positive invariant set for \eqref{eq: LTI+NN} using semidefinite programming. 

To find a positive invariant set for the closed-loop system, we consider the candidate set $\mathcal{E}=\{x \mid \|x\|_{\infty} \leq \epsilon\}$.  We first over approximate the one-step reachable set $f_{cl}(\mathcal{E})$ by the polytope $\mathcal{P} = \{x \mid H x \leq h\}$, $H \in \mathbb{R}^{m \times n_x}, h \in \mathbb{R}^{m}$ (see $\S\ref{subsection:Closed-Loop Reachability Analysis}$). To do this, we form the following $m$ SDPs
\begin{alignat}{2} \label{eq: SDP over abstracted network 2}
& \mathrm{minimize} && \quad h_i \\
& \mathrm{subject \ to} && \quad M_{\mathrm{in}}(P) + M_{\mathrm{mid}}(Q) + M_{\mathrm{out}}(S_i) \preceq 0 \notag \\
& && \quad (P,Q,h_i) \in \mathcal{P}_{\mathcal{E}} \times \mathcal{Q}_{\phi} \times \mathbb{R}, \notag
\end{alignat}
where
$$
S_i = \begin{bmatrix}
0 & 0 & A^\top H^\top e_i \\ 0 & 0 & B^\top H^\top e_i \\ e_i^\top H A & e_i^\top H B & -2e_i^\top h
\end{bmatrix} \ i=1,\cdots,m.
$$
With this choice of $S_i$, it is not difficult to show that the feasibility of the LMIs in \eqref{eq: SDP over abstracted network 2} implies $f_{cl}(\mathcal{E}) \subseteq \mathcal{P}$, and therefore, \eqref{eq: SDP over abstracted network 2} finds the smallest $\mathcal{P}$ that encloses $f_{cl}(\mathcal{E})$. Then, $\mathcal{E}$ is positively invariant if $\mathcal{P} \subseteq \mathcal{E}$.

For the numerical experiment, we first consider a 2D system
\begin{align}
x_{t+1} \!=\! 1.2 \begin{bmatrix}
1 & 1\\  0 & 1
\end{bmatrix} x_t \!+\! \begin{bmatrix}
1 \\ 0.5
\end{bmatrix} u_t,
\end{align}
subject to the state and input constraints $x_t \in \mathcal{X}=\{ x\mid \|x\|_{\infty} \leq 5\}$ and $u \in \mathcal{U}=\{u \mid \|u\|_{\infty} \leq 1\}$. We are interested in stabilizing the system by solving the finite horizon problem
\begin{alignat}{2}
&\mathrm{minimize} &&\sum_{t=0}^{T} \|x_t\|_2^2 + u_t^2 \\ \notag
&\text{s.t.} \ &&(x_t,u_t) \in \mathcal{X} \times \mathcal{U} \ \  t=0,\cdots,T, \ x_0 = x,
\end{alignat}
and choosing the control law as $\mu_{MPC}(x)=u_0^\star$. For generating the training data, we compute $\mu_{MPC}(x)$ at 6284 uniformly chosen random points from the control invariant set. We then train a neural network with two inputs, one output, and two hidden layers with 32 and 16 neurons, respectively using the mean-squared loss. In Figure \ref{fig: mpc and nn policy}, we plot the explicit MPC control law as well as its approximation by the neural network.

In Figure \ref{fig: nn_mpc_1}, we plot the largest invariant set $\mathcal{E}$ that we could find, which is $\mathcal{E}= \{x \mid \|x\|_{\infty} \leq 0.65\}$. In this figure, we also plot the output reachable sets for the first four time steps, starting from the initial set $\mathcal{E}$, as well as their over-approximations by $\deepsdp$.

\section{Conclusions} \label{sec: Conclusions}
\mahyar{We proposed a semidefinite programming framework for robustness analysis and safety verification of feed-forward fully-connected neural networks with general activation functions. Our main idea is to abstract the nonlinear activation functions by quadratic constraints that are known to be satisfied by all possible input-output instances of the activation functions. We then showed that we can analyze the abstracted network via semidefinite programming. We conclude this paper with several future directions.

First, a notable advantage of the proposed SDP compared to other convex relaxations is the relative tightness of the bounds. In particular, coupling all pairs of neurons in the network (repeated nonlinearities) can considerably reduce conservatism. However, coupling all neurons is not feasible for even medium-sized networks as the number of decision variables would scale quadratically with the number of neurons. Nevertheless, our numerical experiments show that most of these pair-wise couplings of neurons are redundant and do not tighten the bounds. It would be interesting to develop a method that can decide \emph{a priori} that coupling which pairs of neurons would tighten the relaxation. Second, one of the drawbacks of SDPs is their limited scalability in general. Exploiting the structure of the problem (e.g. sparsity patterns induced by the network strucrure) to reduce the computational complexity would be an important future direction. Third, we have only considered fully-connected networks in this paper. It would be interesting to extend the results to other architectures. Finally, incorporating the proposed framework in training neural networks with desired robustness properties would be another important future direction.}

\appendix

\section{Appendix}

\subsection{Proof of Proposition \ref{prop: QC for hyperrect}} \label{prop: QC for hyperrect proof}
The inequality $\underline{x} \leq x \leq \bar{x}$ is equivalent to $n_x$ quadratic inequalities of the form
%
$(x_i-\underline{x}_i) (\bar{x}_i-x_i) \geq 0 \quad i=1,\cdots,n_x$. 
%
Multiplying both sides of with $\Gamma_{i} \geq 0$, summing over $i=1,\cdots,n_x$, and denoting $\Gamma=\operatorname{diag}(\gamma_1,\cdots,\gamma_{n_x})$ yields the claimed inequality. $\square$
\subsection{QCs for Polytopes, Zonotopes, and Ellipsoids} \label{QCs for Polytopes, Zonotopes, and Ellipsoids}
\subsubsection{Polytopes} For every vector $x$ satisfying $Hx \leq h$, we have
%
$(H_i^\top x - h_i)(H_j^\top x - h_j) \geq 0, \ i \neq j$,
%
where $H_i^\top$ is the $i$-th row of $H$. These inequalities imply
\begin{align*}
\sum_{1 \leq i,j \leq m} \Gamma_{ij}(H_i^\top x \!-\! h_i)(H_j^\top x \!-\! h_j) \geq 0,
\end{align*}
where $\Gamma_{ij} = \Gamma_{ji} \geq 0,\ i \neq j, \ \Gamma_{ii}=0$. 
%
The preceding inequality is equivalent to \eqref{eq: QC for polytope}.
Now suppose the set $\{x \mid Hx \geq h\}$ is empty. Then
\begin{align*}
\mathcal{X} =  \{x \mid (H_i^\top x-h_i)^\top (H_j^\top x-h_j) \geq 0, \ i \neq j\}.
\end{align*}
To show this set equality define $\mathcal{X}_{Q}$ as the set on the right-hand side. We have $\mathcal{X} \subset \mathcal{X}_{Q}$. To show $\mathcal{X}_{Q} \subset \mathcal{X}$, suppose $x \in \mathcal{X}_Q$, implying that either $H_i^\top x - h_i \leq 0$ for all $i$ or $H_i^\top x - h_i \geq 0$ for all $i$. But the latter cannot happen since the set $\{x \mid Hx \geq h\}$ is empty. Therefore, we have $H_{i}^\top x - h_{i} 
\leq 0 \text{ for all } i$.


\medskip

\subsubsection{Zonotopes} By multiplying both sides of \eqref{eq: QC for zonotopes} by $[\lambda^\top  \ 1]$ and $[\lambda^\top  \ 1]^\top$, respectively, and noting that $x = x_c + A \lambda$ we obtain
\begin{gather*}
\begin{bmatrix}
	x \\ 1
	\end{bmatrix}^\top
	P 
	\begin{bmatrix}
	x \\ 1
	\end{bmatrix} \geq  \begin{bmatrix}
\lambda \\ 1
\end{bmatrix}^\top\begin{bmatrix}
-2\Gamma & \Gamma \mathrm{1}_m \\ -\mathrm{1}_m^\top \Gamma  & 0
\end{bmatrix} \begin{bmatrix}
\lambda \\ 1
\end{bmatrix} \geq 0, \notag
\end{gather*}
where the right inequality follows from the fact that $\lambda \in [0,1]^m$, hence satisfying the QC of Proposition \ref{prop: QC for hyperrect}. 

\subsubsection{Ellipsoids} Any $x \in \mathcal{X}$ satisfies $\mu(1-(Ax+b)^\top (Ax+b)) \geq 0$ for $\mu \geq 0$. The latter inequality is equivalent to \eqref{eq: Ellipsoid}.
$\square$

\subsection{Proof of Lemma \ref{lemma: repeated nonlinearities}} \label{lemma: repeated nonlinearities proof}
For any distinct pairs $(x_i,\varphi(x_i))$ and $(x_j,\varphi(x_j))$, $1\leq i<j \leq n$, we can write the slope restriction inequality in \eqref{eq: quadratic constraint 0} as
\begin{align*} 
\begin{bmatrix}
x_i-x_j \\ \varphi(x_i)-\varphi(x_j)
\end{bmatrix}^\top
\begin{bmatrix}
-2\alpha \beta & \alpha+\beta \\ \alpha+\beta & -2
\end{bmatrix}
\begin{bmatrix}
x_i-x_j \\ \varphi(x_i)-\varphi(x_j)
\end{bmatrix} \geq 0.
\end{align*}
By multiplying both sides by $\lambda_{ij} \geq 0$, we obtain
\begin{align*}  
\begin{bmatrix}
x \\ \phi(x) \\ 1
\end{bmatrix}^\top
\begin{bmatrix}
-2\alpha \beta E_{ij} \lambda_{ij} & (\alpha+\beta) E_{ij}  \lambda_{ij} & 0 \\ (\alpha+\beta) E_{ij}  \lambda_{ij} & -2E_{ij} \lambda_{ij} & 0 \\ 0 & 0 & 0
\end{bmatrix}
\begin{bmatrix}
x \\ \phi(x) \\ 1
\end{bmatrix} \geq 0,
\end{align*}
where $E_{ij}=(e_i-e_j)(e_i-e_j)^\top$ and $e_i \in \mathbb{R}^n$ is the $i$-th unit vector in $\mathbb{R}^n$. Summing over all $1\leq i<j \leq n$ will yield the desired result. 

\subsection{Proof of Lemma \ref{lem: QC for relu}} \label{lem: QC for relu proof}
Consider the equivalence in \eqref{eq: relu qcs} for the $i$-th coordinate of $y = \max(\alpha x,\beta x), \ x\in \mathbb{R}^n$:
\begin{align*}  
(y_i-\alpha x_i)(y_i-\beta x_i)=0, \ y_i \geq \beta x_i, \quad y_i \geq \alpha x_i.
\end{align*}
Multiplying these constraints by $\lambda_i \in \mathbb{R}$, $\nu_i \in \mathbb{R}_{+}$, and $\eta_i \in \mathbb{R}_{+}$, respectively, and adding them together, we obtain
\begin{align*}  
\begin{bmatrix}
x_i \\ y_i \\ 1
\end{bmatrix}^\top  \begin{bmatrix}
-2\alpha \beta \lambda_i & (\alpha+\beta)\lambda_i & -\beta \nu_i -\alpha \eta_i \\  (\alpha+\beta)\lambda_i & -2\lambda_i & \nu_i+\eta_i\\  -\beta \nu_i -\alpha \eta_i & \nu_i+\eta_i & 0
\end{bmatrix}\begin{bmatrix}
x_i \\ y_i \\ 1
\end{bmatrix} \geq 0.
\end{align*}
Substituting $x_i = e_i^\top x$ and $y_i = e_i^\top y$, where $e_i$ is the $i$-th unit vector in $\mathbb{R}^n$, and rearranging terms, we get
\begin{align}  \label{lem: QC for relu 4}
\begin{bmatrix}
x \\ y \\ 1
\end{bmatrix}^\top  Q_i \begin{bmatrix}
x \\ y \\ 1
\end{bmatrix} \geq 0, \ i=1,\cdots,n,
\end{align}
where
\begin{align*}
Q_i = \begin{bmatrix}
-2\alpha \beta \lambda_i & (\alpha+\beta)\lambda_i e_i e_i^\top & (-\beta \nu_i -\alpha \eta_i)e_i \\  (\alpha+\beta)\lambda_i e_i & -2\lambda_i e_i & (\nu_i+\eta_i)e_i\\  (-\beta \nu_i -\alpha \eta_i)e_i & (\nu_i+\eta_i)e_i & 0
\end{bmatrix}.
\end{align*}
Furthermore, since $y_i=\max(\alpha x_i,\beta x_i)$ is slope-restricted in $[\alpha,\beta]$, by Lemma \ref{lemma: repeated nonlinearities} we can write
\begin{align}  \label{lem: QC for relu 5}
\begin{bmatrix}
x \\ y \\ 1
\end{bmatrix}^\top  \begin{bmatrix}
- 2 \alpha \beta T & (\alpha + \beta)T & 0 \\ (\alpha + \beta)T & -2T & 0 \\ 0 & 0 & 0
\end{bmatrix}  \begin{bmatrix}
x \\ y \\ 1
\end{bmatrix} \geq 0.
\end{align}
Summing \eqref{lem: QC for relu 4} over all $i=1,\cdots,n$ and adding the result to \eqref{lem: QC for relu 5} would yield \eqref{lem: QC for relu 1}. $\square$

\subsection{Proof of Lemma \ref{lem: local QC for relu}} \label{lem: local QC for relu proof}
Consider the relation $y=\max(\alpha x,\beta x)$. For active neurons, $i \in \mathcal{I}^{+}$, we can write
\begin{align}  
(y_i-\beta x_i)(y_i-\beta x_i)=0, \ y_i = \beta x_i, \quad y_i \geq \alpha x_i. \notag
\end{align}
Similarly, for inactive neurons, $i \in \mathcal{I}^{-}$,we can write
\begin{align} 
(y_i-\alpha x_i)(y_i-\alpha x_i)=0, \ y_i \geq  \beta x_i, \quad y_i = \alpha x_i. \notag
\end{align}
Finally, for unknown neurons, $i \in \mathcal{I}^{\pm}$, we can write
\begin{align}  
(y_i-\alpha x_i)(y_i-\beta x_i)=0, \ y_i \geq \beta x_i, \quad y_i \geq \alpha x_i. \notag
\end{align}
A weighted combination of the above constraints yields
\begin{align} \label{lem: local QC for relu proof 4}
\sum_{i=1}^{n}\lambda_i (y_i\!-\!\alpha_i x_i)(y_i\!-\!\beta_i x_i) \!+\! \nu_i (y_i\!-\!\beta_i x_i) \!+\! \eta_i (y_i\!-\!\alpha_i x_i) \geq 0
\end{align}
where $\alpha_i = \alpha + (\beta-\alpha)\mathbf{1}_{\mathcal{I}^{+}}(i)$, $\beta_i =\beta -(\beta-\alpha) \mathbf{1}_{\mathcal{I}^{-}}(i)$, $\nu_i \in \mathbb{R}_{+} \text{ for } i \notin \mathcal{I}^{+}$ and $\eta_i \in \mathbb{R}_{+} \text{ for } i  \notin \mathcal{I}^{-}$. Furthermore, since $y_i = \max(\alpha x_i, \beta x_i)$ is slope-restricted on $[\alpha,\beta]$, we can write
\begin{align} \label{lem: local QC for relu proof 5}
-\sum_{i \neq j} \lambda_{ij} (y_j\!-\!y_i\!-\!\alpha (x_j\!-\!x_i))(y_j\!-\!y_i \!-\! \beta(x_j\!-\!x_i))\geq 0.
\end{align}
Adding \eqref{lem: local QC for relu proof 4} and \eqref{lem: local QC for relu proof 5} and rearranging terms would yield the desired inequality. $\square$

\subsection{Proof of Theorem \ref{thm: main result one layer}} \label{thm: main result one layer proof}
Consider the identity $x^1 = \phi(W^0 x^0 + b^0)$. Using the assumption that  $\phi$ satisfies the quadaratic constraint defined by $\mathcal{Q}_{\phi}$ on $\mathcal{Z}$, $x^0,x^1$ satisfy the QC 
%
%
\begin{align} \label{thm: hyperplance one layer 8.5}
\begin{bmatrix}
x^0 \\ x^1 \\ 1
\end{bmatrix}^\top 
\underbrace{
	\begin{bmatrix}
	W^0 & 0 & b^0 \\ 0 & I_{n_1} & 0 \\ 0 & 0 & 1
	\end{bmatrix}^\top Q \begin{bmatrix}
	W^0 & 0 & b^0 \\ 0 & I_{n_1} & 0 \\ 0 & 0 & 1
	\end{bmatrix}}_{M_{\mathrm{mid}}(Q)}\begin{bmatrix}
x^0 \\ x^1 \\ 1
\end{bmatrix} \geq 0,
%
\end{align}
for any $Q \in \mathcal{Q}_{\phi}$ and all $x^0 \in \mathcal{X}$. By assumption $\mathcal{X}$ satisfies the QC defined by $\mathcal{P}_{\mathcal{X}}$, implying that for any $P  \in \mathcal{P}_{\mathcal{X}}$,
%
%
\begin{align} \label{thm: hyperplance one layer 10}
\begin{bmatrix}
x^0 \\ x^1 \\ 1
\end{bmatrix}^\top 
\underbrace{
	\begin{bmatrix}
	I_{n_0} & 0 \\ 0 & 0 \\ 0 & 1
	\end{bmatrix}^\top P \begin{bmatrix}
	I_{n_0} & 0 & 0 \\ 0 & 0 & 1
	\end{bmatrix}}_{M_{\mathrm{in}}(P)}\begin{bmatrix}
x^0 \\ x^1 \\ 1
\end{bmatrix}  \geq 0,
\end{align}
for all $x^0 \in \mathcal{X}$. Suppose \eqref{thm: main result one layer 2} holds for some $(P,Q) \in \mathcal{P}_{\mathcal{X}} \times \mathcal{Q}_{\phi}$. By left- and right- multiplying both sides of \eqref{thm: main result one layer 2} by $[{x^0}^\top \ {x^1}^\top \ 1]$ and $[{x^0}^\top \ {x^1}^\top \ 1]^\top$, respectively, we obtain
\begin{gather*}
\underbrace{    \begin{bmatrix}
	x^0 \\ x^1 \\ 1
	\end{bmatrix}^\top M_{\mathrm{in}}(P)     \begin{bmatrix}
	x^0 \\ x^1 \\ 1
	\end{bmatrix}}_{\geq 0 \text{ for all $x^0 \in \mathcal{X}$ by } \eqref{thm: hyperplance one layer 10}}  + \underbrace{\begin{bmatrix}
	x^0 \\ x^1 \\ 1 \end{bmatrix}^\top M_{\mathrm{mid}}(Q) \begin{bmatrix}
	x^0 \\ x^1 \\ 1 \end{bmatrix}}_{\geq 0 \text{ for all $x^0 \in  \mathcal{X}$ by } \eqref{thm: hyperplance one layer 8.5}}  \\ + \begin{bmatrix}
x^0 \\ x^1 \\ 1 \end{bmatrix}^\top  M_{\mathrm{out}}(S) \begin{bmatrix}
x^0 \\ x^1 \\ 1 \end{bmatrix} \leq 0. \notag
\end{gather*}
Therefore, the last term on the left-hand side must be nonpositive for all $x^0 \in \mathcal{X}, \ x^1=\phi(W^0 x^0 + b^0)$, or, equivalently,
\begin{align*}
\begin{bmatrix}
x^0 \\ x^1 \\ 1
\end{bmatrix}^\top \begin{bmatrix}
I_{n_0} & 0 & 0 \\ 0 & {W^1}^\top & 0 \\ 0 & {b^1}^\top & 1
\end{bmatrix} S \begin{bmatrix}
I_{n_0} & 0 & 0 \\ 0 & W^1 & b^1 \\ 0 & 0 & 1
\end{bmatrix} \begin{bmatrix}
x^0 \\ x^1 \\ 1
\end{bmatrix} \leq 0.
\end{align*}
Using the relations $x^0=x$ and $f(x) = W^1 x^1 + b^1$, the above inequality is the desired inequality in \eqref{thm: main result multi layer 2}.
$\square$

\subsection{Proof of Theorem \ref{thm: main result multi layer}} \label{thm: main result multi layer proof}
Recall the definition $\mathcal{Z}= \{\bA\bx +\bb \mid x \in \mathcal{X}\}$. Since $\phi$ satisfies the QC defined by $\mathcal{Q}_{\phi}$ on $\mathcal{Z}$, for any $Q \in \mathcal{Q}_{\phi}$,  we have
%
%
\begin{align} \label{thm: hyperplance multi layer 8}
\begin{bmatrix}
\bbx \\ 1
\end{bmatrix}^\top 
\underbrace{
	\begin{bmatrix}
	\bA & \bb \\ \bB & 0 \\ 0 & 1
	\end{bmatrix}^\top Q \begin{bmatrix}
	\bA & \bb \\ \bB & 0 \\ 0 & 1
	\end{bmatrix}}_{M_{\mathrm{mid}}(Q)} \begin{bmatrix}
\bbx \\ 1
\end{bmatrix} \geq 0 \text{ for all } x^0 \in \mathcal{X},
\end{align}
By assumption $\mathcal{X}$ satisfies the QC defined by $\mathcal{P}_{\mathcal{X}}$.
%
Using the relation $x^0 = \bE^0 \bbx$, for any $P \in \mathcal{P}_{\mathcal{X}}$ it holds that
\begin{align} \label{thm: hyperplance multi layer 10}
\begin{bmatrix}
\bbx \\ 1
\end{bmatrix}^\top 
\underbrace{
	\begin{bmatrix}
	\bE^0 & 0 \\  0 & 1
	\end{bmatrix}^\top P \begin{bmatrix}
	\bE^0 & 0 \\  0 & 1
	\end{bmatrix}}_{M_{\mathrm{in}}(P)}\begin{bmatrix}
\bbx \\ 1
\end{bmatrix}  \geq 0  \text{ for all } x^0 \in \mathcal{X}.
\end{align}
Suppose the LMI in \eqref{thm: main result multi layer 1} holds for some $(P,Q) \in \mathcal{P}_{\mathcal{X}} \times \mathcal{Q}_{\phi}$. By left- and right- multiplying both sides of \eqref{thm: main result one layer 1} by $[\bbx^\top \ 1]$ and $[\bbx^\top \ 1]^\top$, respectively, we obtain
\begin{gather*}
\underbrace{\begin{bmatrix}
	\bbx \\ 1
	\end{bmatrix}^\top \!M_{\mathrm{in}}(P)\! \begin{bmatrix}
	\bbx \\ 1
	\end{bmatrix}}_{\geq 0 \text{ by } \eqref{thm: hyperplance multi layer 10}} +  \underbrace{\begin{bmatrix}
	\bbx \\ 1
	\end{bmatrix}^\top \!M_{\mathrm{mid}}(Q)\! \begin{bmatrix}
	\bbx \\ 1
	\end{bmatrix}}_{\geq 0 \text{ by } \eqref{thm: hyperplance multi layer 8}} \\ +     \begin{bmatrix}
\bbx \\ 1
\end{bmatrix}^\top \!M_{\mathrm{out}}(S)\! \begin{bmatrix}
\bbx \\ 1
\end{bmatrix} \leq 0. \notag
\end{gather*}
Therefore, the last quadratic term must be nonpositive for all $x^0 \in \mathcal{X}$, from where we can write
\begin{align*}
\begin{bmatrix}
\bbx \\ 1
\end{bmatrix}^\top \begin{bmatrix}
\bE^0 & 0 \\ W^{\ell}\bE^{\ell} & b^{\ell} \\ 0 & 1
\end{bmatrix}^\top S \begin{bmatrix}
\bE^0 & 0 \\ W^{\ell}\bE^{\ell} & b^{\ell} \\ 0 & 1
\end{bmatrix} \begin{bmatrix}
\bbx \\ 1
\end{bmatrix} \leq 0  \text{ for all } x^0 \in \mathcal{X}.
\end{align*}
Using the relations $x^0=\bE^0 \bbx $ and $f(x) = W^{\ell} \bE^{\ell} \bbx + b^{\ell}$ from \eqref{eq: multi layer neural net}, the above inequality can be written as
$$
\begin{bmatrix}
x^0 \\ f(x^0) \\1
\end{bmatrix}^\top S \begin{bmatrix}
x^0 \\ f(x^0) \\1
\end{bmatrix} \leq 0, \text{ for all } x^0 \in \mathcal{X}.
$$
$\square$

\subsection{More Visualizations} \label{app: More Visualizations}
In Figure \ref{fig:Experiment_2}, we show the effect of the number of hidden neurons on the quality of approximation for a single-layer network, and in Figure \ref{fig:Experiment_3}, we change the perturbation size.


\begin{figure}
	\centering
	\includegraphics[width=1\linewidth]{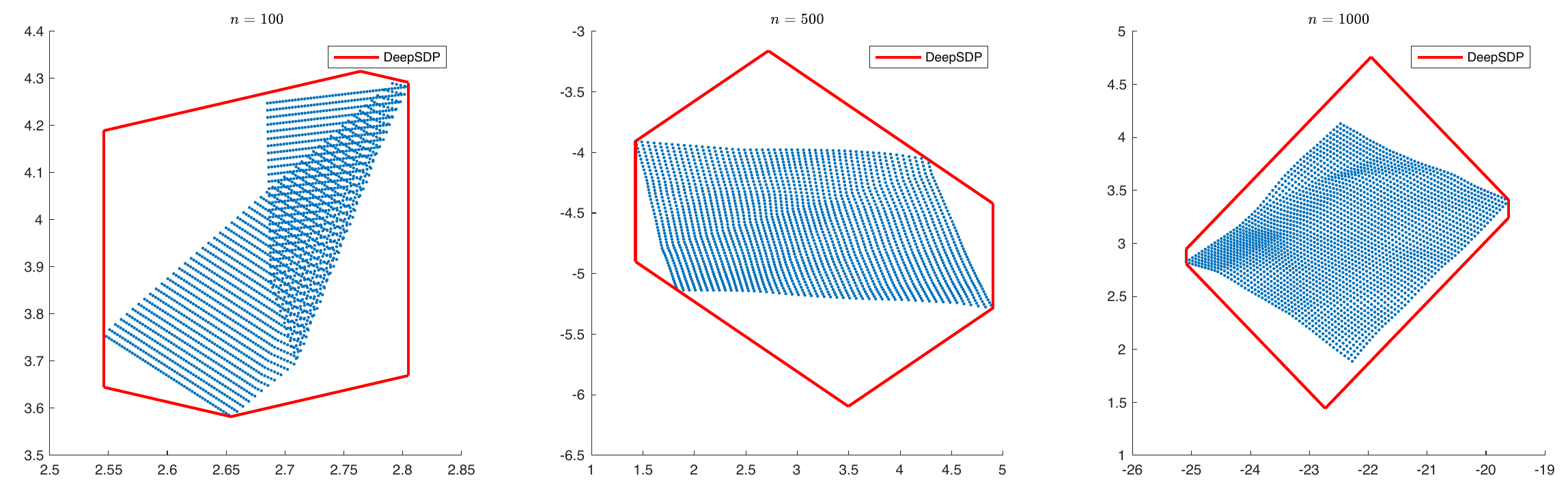}
	\caption{\small The effect of the number of hidden neurons on the over-approximation quality of the SDP for a one-layer neural network with $100$ (left), $500$ (middle), and $1000$ hidden nuerons (right). The activation function is $\relu$. Quadratic constraints for repeated nonlinearity are not included.}
	\label{fig:Experiment_2}
\end{figure}
\begin{figure}
	\centering
	\includegraphics[width=1\linewidth]{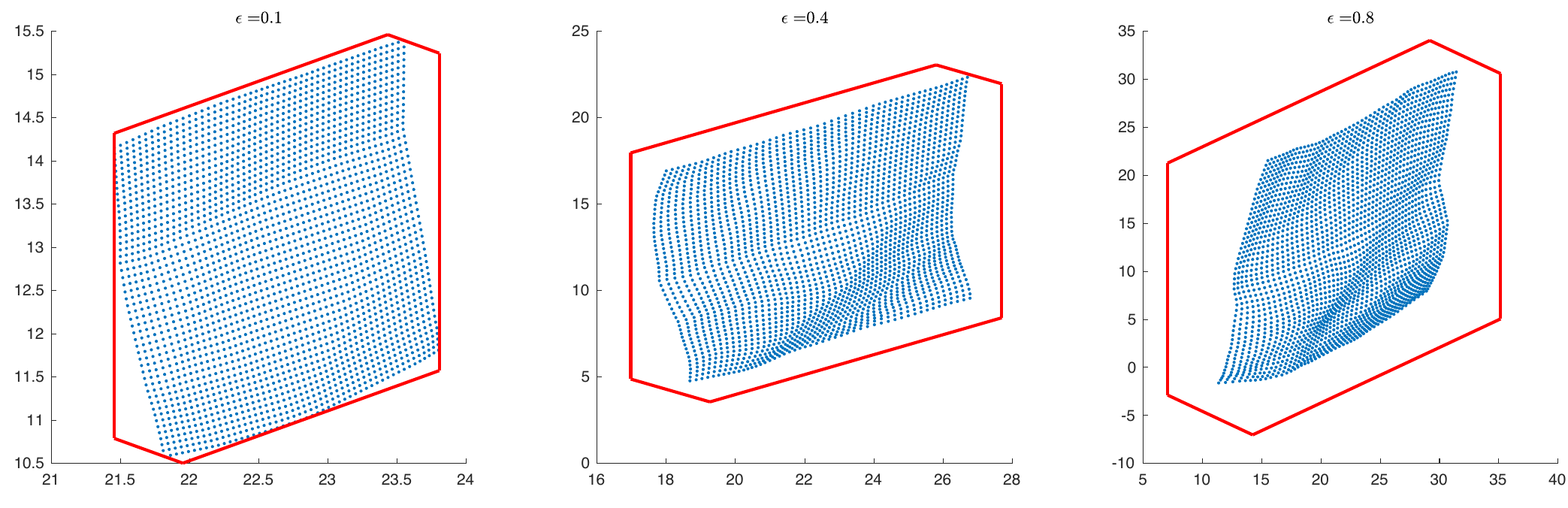}
	\caption{\small The effect of $\epsilon$ (the $\ell_{\infty}$ norm of the input set) on the over-approximation quality of the SDP for $\epsilon=0.1$ (left), $\epsilon=0.4$ (middle), and $\epsilon=0.8$ (right). The network architecture is 2-500-2 with $\relu$ activation functions. QCs for repeated nonlinearity are not included.}
	\label{fig:Experiment_3}
\end{figure}

\bibliographystyle{ieeetr}
\bibliography{Refs}

\vspace{-10mm}

\begin{IEEEbiography}[{\includegraphics[width=1in,height=1.25in,clip,keepaspectratio]{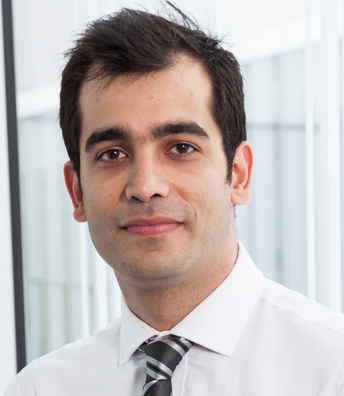}}]{Mahyar Fazlyab}
	(S'13) received his Ph.D. in Electrical and Systems Engineering from the University of Pennsylvania, Philadelphia, PA, USA, in 2018. He was also a postdoctoral fellow in the ESE Department at UPenn from 2018 to 2020. He will join the Department of Electrical and Computer Engineering and Mathematical Institute for Data Science (MINDS) at Johns Hopkins University as an Assistant Professor in 2021. His research interests are at the intersection of optimization, control, and machine learning.
	Dr. Fazlyab won the Joseph and Rosaline Wolf Best Doctoral Dissertation Award in 2019, awarded by the Department of Electrical and Systems Engineering at the University of Pennsylvania.
\end{IEEEbiography}

\begin{IEEEbiography}[{\includegraphics[width=1in,height=1.25in,clip,keepaspectratio]{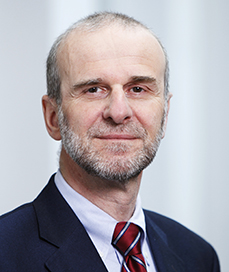}}]{Manfred Morari}
	(F'05) received the Diploma degree in chemical engineering from ETH Z\"{u}rich, Z\"{u}rich, Switzerland, and the Ph.D. degree in chemical engineering from the University of Minnesota, Minneapolis, MN, USA.
	He was a Professor and the Head of the Department of Information Technology and Electrical
	Engineering, ETH Z\"{u}rich. He was the McCollumCorcoran Professor of chemical engineering and the
	Executive Officer of control and dynamical systems with the California Institute of Technology
	(Caltech), Pasadena, CA, USA. He was a Professor at the University
	of Wisconsin, Madison, WI, USA. He is currently with the University of
	Pennsylvania, Philadelphia, PA, USA. He supervised more than 80 Ph.D.
	students.
	
	Dr. Morari is a fellow of AIChE, IFAC, and the U.K. Royal Academy of
	Engineering. He is a member of the U.S. National Academy of Engineering.
	He was a recipient of numerous awards, including Eckman, Ragazzini, and
	Bellman Awards from the American Automatic Control Council (AACC);
	Colburn, Professional Progress, and CAST Division Awards from the American Institute of Chemical Engineers (AIChE); Control Systems Award and
	Bode Lecture Prize from IEEE; Nyquist Lectureship and Oldenburger Medal
	from the American Society of Mechanical Engineers (ASME); and the IFAC
	High Impact Paper Award. He was the President of the European Control
	Association. He served on the technical advisory boards of several major
	corporations.

\end{IEEEbiography}

\begin{IEEEbiography}[{\includegraphics[width=1in,height=1.25in,clip,keepaspectratio]{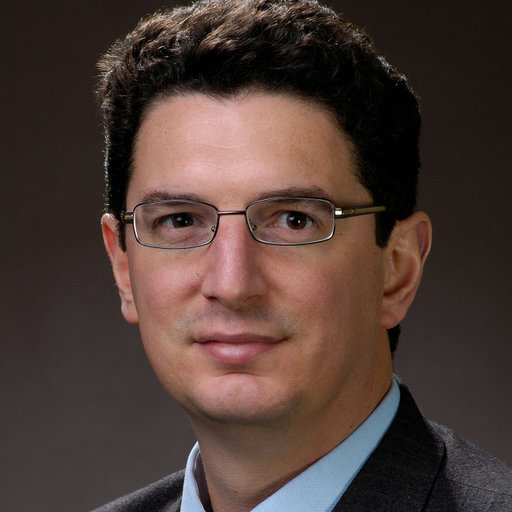}}]{George J. Pappas}
	(S'90--M'91--SM'04--F'09) received the Ph.D. degree in electrical engineering and computer sciences from the University of California, Berkeley, CA, USA, in 1998. He is currently the Joseph Moore Professor and Chair of the Department of Electrical and Systems Engineering, University of Pennsylvania, Philadelphia, PA, USA. He also holds a secondary appointment with the
	Department of Computer and Information Sciences and the Department of Mechanical Engineering and Applied Mechanics. He is a member of the GRASP Lab and the PRECISE Center. He had previously served as the Deputy
	Dean for Research with the School of Engineering and Applied Science. His
	research interests include control theory and, in particular, hybrid systems,
	embedded systems, cyberphysical systems, and hierarchical and distributed
	control systems, with applications to unmanned aerial vehicles, distributed
	robotics, green buildings, and biomolecular networks. Dr. Pappas has received
	various awards, such as the Antonio Ruberti Young Researcher Prize, the
	George S. Axelby Award, the Hugo Schuck Best Paper Award, the George
	H. Heilmeier Award, the National Science Foundation PECASE award, and
	numerous best student papers awards at ACC, CDC, and ICCPS.
\end{IEEEbiography}

\end{document}